\documentclass[11pt,reqno]{amsart}
\usepackage[T1]{fontenc}
\usepackage{color, amssymb, amsmath}
\usepackage{cite}
\usepackage{xcolor}
\definecolor{MyLinkColor}{rgb}{0,0,0.4}

\newcommand{\R}{{\mathbb R}}

\newcommand{\Z}{{\mathbb Z}}

\newcommand{\s}{\mathbb S}

\newcommand{\kH}{\mathcal{H}}

\newcommand{\kL}{\mathcal{L}}

\newcommand{\wh}{\widehat}
\newcommand{\wt}{\widetilde}

\newcommand{\re}{\mathop{\rm Re}\nolimits}

\newcommand{\ov}{\overline}
\newcommand{\p}{\partial}

\newcommand{\e}{\varepsilon}
\newcommand{\rd}{\mathrm{d}}
\newcommand{\0}{\Omega}

\newtheorem{thm}{Theorem}[section]
\newtheorem{prop}[thm]{Proposition}
\newtheorem{lemma}[thm]{Lemma}
\newtheorem{cor}[thm]{Corollary}
\newtheorem{rem}[thm]{Remark}

\theoremstyle{remark}

\numberwithin{equation}{section} 
\usepackage[colorlinks=true,linkcolor=MyLinkColor,citecolor=MyLinkColor]{hyperref}

\newcommand{\tblue}{\color{black}}
\newcommand{\tred}{\color{black}}

\begin{document}

 \author[B.--V. Matioc]{Bogdan--Vasile Matioc}
\address{Fakult\"at f\"ur Mathematik, Universit\"at Regensburg,   93053 Regensburg, Germany.}
\email{bogdan.matioc@ur.de}

 \author[L. Roberti]{Luigi Roberti}
\address{Fakult\"at f\"ur Mathematik, Universit\"at Wien,  Oskar--Morgenstern--Platz 1, 1090, Vienna, Austria.}
\email{luigi.roberti@univie.ac.at}

\author[Ch. Walker]{Christoph Walker}
\address{Leibniz Universit\"at Hannover,
Institut f\"ur Angewandte Mathematik,
Welfengarten 1,
30167 Hannover,
Germany.}
\email{walker@ifam.uni-hannover.de}

\title[Quasilinear parabolic equations in critical spaces]{Quasilinear parabolic equations with  superlinear  nonlinearities in critical spaces}

\thanks{}

\begin{abstract}
Well-posedness in time-weighted spaces for  quasilinear (and semilinear) parabolic evolution equations $u'=A(u)u+f(u)$ is established 
in a certain critical case of  strict inclusion $\mathrm{dom}(f)\subsetneq \mathrm{dom}(A)$ for the domains of the  (superlinear) function $u\mapsto f(u)$ and  the quasilinear part $u\mapsto A(u)$. 
Based upon regularizing effects of parabolic equations, it is proven that the solution map generates a semiflow  in a critical intermediate space.  
The applicability of the abstract results is demonstrated by several examples including a model for atmospheric flows and semilinear and quasilinear evolution equations with scaling invariance for which well-posedness in the critical scaling invariant intermediate spaces is shown.
\end{abstract}

\subjclass[2020]{35K59; 35K58; 35B35; 76U60}
\keywords{Quasilinear parabolic equations; semilinear parabolic equations;  critical spaces;   scaling invariance; atmospheric flows}

\maketitle

\pagestyle{myheadings}
\markboth{\sc{B.-V.~Matioc, L. Roberti, $\&$ Ch. Walker}}{\sc{Quasilinear parabolic equations  in critical spaces}}

\section{Introduction}

Solvability of abstract quasilinear parabolic problems 
\begin{equation}\label{QCP}
u'=A(u)u+f(u)\,,\quad t>0\,,\qquad u(0)=u_0\,,
\end{equation} 
where $A(u)$ is the generator of an analytic semigroup on an appropriate Banach space, has a long history and is well-established by means of 
semigroup or maximal regularity theory, see e.g. \cite{ AT88, Am88,Amann_Teubner,ClementLi,CS01,G88,Lu84,Lu85,Lu85b,P02,PS16,PSW18,MW_MOFM20}  (and e.g. \cite{An90,DaPL88,DaP96,DaPG79,Lu87,L95} for even fully nonlinear problems). 
In most of the considered cases, the involved nonlinearities $u\mapsto f(u)$ and~${u\mapsto A(u)}$ have the same domain of definition.
In this article we  focus on the case  of strict inclusions $\mathrm{dom}(f)\subsetneq \mathrm{dom}(A)$,
 which yields more flexibility when applying the theory to concrete problems. 
 To  handle the regularity gap induced by  such a strict inclusion we shall use time-weighted spaces.

To set the stage, let   $E_0$ and $E_1$ be Banach spaces  over $\mathbb{K}\in \{\R,\mathbb{C}\}$ with   continuous  and dense embedding 
$E_1\hookrightarrow E_0$.
For each $\theta\in (0,1)$, we denote by $(\cdot,\cdot)_\theta$   an arbitrary admissible interpolation functor of 
 exponent $\theta$  (see \cite[I.~Section~2.11]{LQPP}) and define~${E_\theta:= (E_0,E_1)_\theta}$ as the corresponding  interpolation space with norm~\mbox{$\|\cdot\|_\theta$}.
 
We fix  exponents
\begin{subequations}\label{ASS}
\begin{equation}\label{AS1}
\quad 0<\gamma< \beta<\alpha< \xi<1\qquad \text{ with }\qquad   q:=\frac{1+\gamma-\alpha}{\xi-\alpha}>1\,.
\end{equation}
Moreover, we assume that there are  interpolation functors $\{\cdot,\cdot\}_{\alpha/\xi}$ and $\{\cdot,\cdot\}_{\gamma/\eta}$ of  exponents~$\alpha/\xi$  and~$\gamma/\eta$ for  $\eta\in\{\alpha,\beta,\xi\}$, respectively, such that
\begin{equation}\label{Vor3q}
   E_\alpha\doteq \{E_0,E_\xi\}_{\alpha/\xi} \,, \qquad  E_\gamma\doteq \{E_0,E_\eta\}_{\gamma/\eta} \,, \quad \eta\in\{\beta,\alpha,\xi\}\,.
\end{equation} 
For the quasilinear part we assume for some open subset {\tblue $O_\beta$ of $E_\beta$} that
\begin{equation}\label{AS2}
A\in C^{1-}\big(O_\beta,\mathcal{H}(E_1,E_0) \big)\,,
\end{equation}
 where  $C^{1-}$ refers to local Lipschitz continuity and $\mathcal{H}(E_1,E_0)$ is the open subset of the bounded linear operators $\kL(E_1,E_0)$  
consisting of generators $A$ of strongly continuous analytic semigroups  $(e^{tA})_{t\ge 0}$ on $E_0$.

 Moreover, denoting by $O_\xi $ the open subset  of $E_\xi$ defined by $O_\xi:=O_\beta\cap E_\xi$, we assume that the semilinear part $f:O_\xi\to E_\gamma$ {\tblue 
is locally Lipschitz continuous in the sense that, for each $R>0$, there is a constant~${N(R)>0}$ such that 
\begin{equation}\label{AS4}
\|f(w)-f(v)\|_{\gamma}\le  N(R)\big[1+\|w\|_{\xi}^{q-1}+\|v\|_{\xi}^{q-1}\big]\big[\big(1+\|w\|_{\xi}+\|v\|_{\xi}\big) \|w-v\|_{\beta}+\|w-v\|_{\xi}\big]
\end{equation}
for $w,\,v\in  O_\xi \cap\bar{\mathbb{B}}_{E_\beta}(0,R)$, 
where $q>1$ is defined in~\eqref{AS1} and $\bar{\mathbb{B}}_{E_\beta}(0,R)$ is the closed ball in $E_\beta$ centered  at $0$ with radius $R$. Of course, ~\eqref{AS4} includes the case that
\begin{equation*}
\|f(w)-f(v)\|_{\gamma}\le  N(R)\big[1+\|w\|_{\xi}^{q-1}+\|v\|_{\xi}^{q-1}\big] \|w-v\|_{\xi}
\end{equation*}
for $w,\,v\in  O_\xi \cap\bar{\mathbb{B}}_{E_\beta}(0,R)$.
}

\end{subequations}

Under assumptions~\eqref{ASS} we shall prove  that problem~\eqref{QCP} is well-posed  in $O_\alpha:=O_\beta\cap E_\alpha$. 
 We emphasize that the semilinear part $f$ needs not be defined on $O_\alpha$ and possibly requires more regularity 
 than the quasilinear part $A$ since $E_\xi\hookrightarrow E_\alpha\hookrightarrow E_\beta$ by~\eqref{AS1}. Before stating the main result  precisely, let us comment on the imposed assumptions. If the nonlinearities~$A$ and~$f$
are defined and Lipschitz continuous on the same set $O_\beta$, a general  and powerful result on the well-posedness of~\eqref{QCP} is  \cite[Theorem~12.1]{Amann_Teubner} 
(for proofs see \cite{Am88} and \cite{MW_MOFM20}).  
 It states that if
\begin{equation*}
(A,f)\in C^{1-}\big(O_\beta,\mathcal{H}(E_1,E_0)\times E_\gamma\big)\,,\qquad u_0\in O_\alpha\,,\quad 0<\gamma\le \beta<\alpha<1\,,
\end{equation*}
 then problem~\eqref{QCP} admits a unique maximal strong solution and the solution map induces a semiflow on $O_\alpha$. 
 Our purpose  herein is to prove a similar result under the assumptions stated above.
In fact, we extend our previous work~\cite{MW_PRSE}, where we  proved the semiflow property  when the second part of assumption~\eqref{AS1} is replaced by the assumption  that the strict inequality
\begin{equation}\label{strict}
 1\leq q<\frac{1+\gamma-\alpha}{ \xi-\alpha}
\end{equation} 
holds, to the critical case $q(\xi-\alpha)=1+\gamma-\alpha$.
 The latter covers important applications  and yields a semiflow in spaces that are optimal in a certain sense. As pointed out {\tblue in  \cite{PW17,PSW18,HvNWV_III}}, for 
evolution equations which possess a scaling invariance, the  space $E_\alpha$ corresponding to the critical case $q(\xi-\alpha)=1+\gamma-\alpha$ 
may be  invariant under  this scaling and is therefore in this sense a critical  space (see Example~\ref{Exam2} and Example~\ref{Exam3} in Section~\ref{Sec:4} for more details).
{\tblue In \cite{PW17,PSW18} the quasilinear problem~\eqref{QCP} is thoroughly  investigated  in the context of  maximal 
$L_p$-regularity for the critical case~\eqref{AS1}, with} an additional structural condition
 on the Banach spaces~$E_0$ and~$E_1$ (i.e. $E_0$ needs to be  a UMD space and~${H_p^1(\R^+,E_0)\cap L_p(\R^+,E_1)\hookrightarrow H_p^{1-\theta}(\R^+,E_\theta)}$ for each $\theta\in [0,1]$). 
{\tblue Recently it has been shown that these structural conditions are in fact not needed for the critical case in the maximal 
$L_p$-regularity setting, see \cite[Section~18.2]{HvNWV_III} for details and also \cite[Section~18.5]{HvNWV_III} for a discussion of the history of the critical case.} {\tred It is worth emphasizing that  our approach covers situations for which maximal regularity does not seem to be available (see Example~\ref{Exam4}).}

In our setting, we impose only the interpolation property~\eqref{Vor3q} as an additional assumption to the subcritical case \eqref{strict}. Regarding this property, we point out:

\begin{rem}\label{R00}
Assumption~\eqref{Vor3q} is automatically satisfied if~$(\cdot,\cdot)_\theta=[\cdot,\cdot]_\theta$ is the complex interpolation functor for every~$\theta\in\{\gamma,\beta,\alpha,\xi\}$, {\tred or  if~$(\cdot,\cdot)_\theta=(\cdot,\cdot)_{\theta,p}$ is the real interpolation functor with parameter $p\in [1,\infty]$ for every~$\theta\in\{\gamma,\beta,\alpha,\xi\}$, or if~$(\cdot,\cdot)_\theta=(\cdot,\cdot)_{\theta,\infty}^0$ is the continuous interpolation functor for every~$\theta\in\{\gamma,\beta,\alpha,\xi\}$.
 This is explicitly stated in~\cite[I.Remarks~2.11.2~(b)]{LQPP} and a consequence of  the reiteration theorems for the complex and the real methods (see~\cite[I.Section~2.8]{LQPP}).}
\end{rem}

For the subcritical regime~\eqref{strict},  well-posedness of \eqref{QCP} is shown in \cite{CPW10} in time-weighted $L_p$-spaces in the framework of maximal \mbox{$L_p$-regularity} assuming that~$f$ satisfies {\tblue a   local Lipschitz continuity property  similar to~\eqref{AS4}} in real interpolation spaces
with $\gamma=0$. We also refer to \cite{LeCroneSimonett20} for a result in the same spirit based on the concept of
continuous maximal regularity in time-weighted spaces and assuming \eqref{AS4} 
 with strict inequality~\eqref{strict} in the scale of continuous interpolation spaces.

We now state our main Theorem~\ref{T:0x} that extends our previous research~\cite{MW_PRSE} on the subcritical case~\eqref{strict} to the critical case~\eqref{AS1}. It is comparable to the above {\tblue mentioned results  \cite{PW17,PSW18,LeCroneSimonett20,HvNWV_III} for} the critical case~\eqref{AS1} outside the setting of maximal regularity and for (almost) arbitrary admissible interpolation functors:

\begin{thm}\label{T:0x}
 Suppose~\eqref{ASS}. 

 \begin{itemize}
{\bf\item[(i)] {\bf\emph{(Existence)}}} Given any $u_0\in O_\alpha$,  there is $t^+=t^+(u_0)\in(0,\infty]$ such that the quasilinear Cauchy problem \eqref{QCP} possesses  a maximal strong  solution
\begin{equation} \label{reg:T0x}
\begin{aligned}
 u(\cdot;u_0)&\in C^1\big((0,t^+),E_0\big)\cap C\big((0,t^+),E_1\big)\cap  C \big([0,t^+),O_\alpha\big)\cap  C^{\alpha-\beta}\big([0,t^+),E_\beta\big)
\end{aligned}
\end{equation}
 with
$$
\lim_{t\to 0}t^{\xi-\alpha}\|u(t;u_0)\|_\xi=0\,.
$$
 
 {\bf \item[(ii)]  {\bf\emph{(Uniqueness)}}}   If 
\begin{equation*} 
 \wt u\in C^1\big((0,T],E_0\big)\cap C\big((0,T],E_1\big) \cap C^{\vartheta}\big([0,T],O_\beta\big) 
\end{equation*}
with
$$
\lim_{t\to 0}t^{\xi-\alpha}\|\wt u(t)\|_\xi=0
$$
is a solution to \eqref{QCP} for some $T>0$ and $\vartheta\in(0,1)$, then    $T< t^+(u_0)$ and~${\wt u=u(\cdot;u_0)}$ on $[0,T]$. \vspace{1mm}

{\bf \item[(iii)] {\em (Continuous dependence)}} The map $(t,u^0)\mapsto u(t;u_0)$ is a semiflow on $O_\alpha$. \vspace{1mm}

{\bf \item[(iv)] {\em (Global existence)}} If  $u([0,t^+(u_0));u_0)$ is relatively compact in $O_\alpha$, then~${t^+(u_0)=\infty}$.\vspace{1mm} 

{\bf \item[(v)] {\em (Blow-up criteria)}} Let $u_0\in O_\alpha$ be such that $t^+(u_0)<\infty$. \vspace{1mm}
\begin{itemize}
\item[{\bf (a)}]  If  $u(\cdot;u_0):[0,t^+(u_0))\to E_\alpha$ is uniformly continuous, then
\begin{equation}\label{boundx} 
 \lim_{t\nearrow t^+(u_0)}\mathrm{dist}_{E_\alpha}\big( u(t;u_0),\partial O_\alpha\big)=0\,.
\end{equation}

\item[{\bf (b)}]  If $u(\cdot;u_0):[0,t^+(u_0))\to E_\beta$ is uniformly H\"older continuous, then
\begin{equation}\label{v1}
 \underset{t\nearrow t^+(u_0)}{\limsup}\|f(u(t;u_0))\|_{0}=\infty\quad\text{ or }\quad \lim_{t\nearrow t^+(u_0)}\mathrm{dist}_{E_\beta}\big( u(t;u_0),\partial O_\beta\big)=0\,.
\end{equation}

\item[{\bf (c)}]  If $E_1$ is compactly embedded in $E_0$, then 
\begin{equation}\label{v2}
\limsup_{t\nearrow t^+(u_0)}\|u(t;u_0)\|_{\theta}=\infty\quad\text{ or }\quad  \lim_{t\nearrow t^+(u_0)}\mathrm{dist}_{E_\alpha}\big( u( [0,t];u_0),\partial  O_\alpha\big)=0
\end{equation}
for each $\theta\in (\alpha,1)$.
\end{itemize}

\end{itemize}
\end{thm}

 {\tblue The local existence and uniqueness result of Theorem~\ref{T:0x} relies on Banach's fixed point theorem, 
 where we point out the technical difficulties induced by the singularity at $t=0$ of $t\mapsto f(u(t))$. 
 These are handled by working in a framework of time-weighted spaces as previously in~\cite{MW_PRSE} for the subcritical case~\eqref{strict}.
}

Of course, Theorem~\ref{T:0x} is  valid also  in the particular  case of   semilinear  parabolic problems
\begin{equation}\label{SCP}
u'=Au+f(u)\,,\quad t>0\,,\qquad u(0)= u_0\,,
\end{equation}
with 
\begin{subequations}\label{Vor}
\begin{equation}\label{Vor1}
A\in \mathcal{H}(E_1,E_0)\,.
\end{equation} 
However,  for semilinear problems Theorem~\ref{T:0x} can be sharpened.
Indeed, we now  impose
 \begin{equation}\label{Vor2}
  0\leq \gamma< \alpha<  \xi\leq  1\,,\qquad  q:= \frac{1+\gamma-\alpha}{\xi-\alpha}>1\,.
\end{equation} 
 Note that \eqref{Vor2} implies in particular that $(\gamma,\xi)\neq(0,1)$. 
Moreover,  if $\xi<1$, we assume that there exists an interpolation functor  $\{\cdot,\cdot\}_{\alpha/\xi}$ of exponent $\alpha/\xi$  and, if~${\gamma>0}$, we assume  for~${\eta\in\{\alpha,\xi\}}{\tred \setminus\{1\}}$ that 
there are  interpolation functors $\{\cdot,\cdot\}_{\gamma/\eta}$ of exponent~$\gamma/\eta$ such that
\begin{equation}\label{Vor3}
   E_\alpha\doteq \{E_0,E_\xi\}_{\alpha/\xi}\,\,\, \text{if $\xi<1$}  \,, \qquad  E_\gamma\doteq \{E_0,E_\eta\}_{\gamma/\eta} \,\,\, \text{if $\gamma>0$\,, \ $\eta\in\{\alpha,\xi\}{\tred \setminus\{1\}}$\,.}
\end{equation} 
 Let $O_\alpha$ be an open subset {\tblue of $E_\alpha $  and put $O_\xi:=O_\alpha\cap E_\xi$. 
 The  nonlinearity~${f:O_\xi\to E_\gamma}$ is supposed to be locally Lipschitz continuous  in the sense that for each $R>0$ there exists a positive constant~$N(R)$ such that 
\begin{equation}\label{Vor5}
\|f(w)-f(v)\|_{\gamma}\le   N(R)\big[1+\|w\|_{\xi}^{q-1}+\|v\|_{\xi}^{q-1}\big]\big[\big(1+\|w\|_{\xi}+\|v\|_{\xi}\big) \|w-v\|_{\alpha}+\|w-v\|_{\xi}\big] 
\end{equation}
for all $w,\,v\in  O_\xi \cap\bar{\mathbb{B}}_{E_\alpha}(0,R)$.}
\end{subequations}
The well-posedness result  in the framework of the semilinear parabolic problem~\eqref{SCP} reads as follows:

\begin{thm}\label{T:2}
 Suppose~\eqref{Vor}. 

\begin{itemize}
{\bf \item[(i)] \emph{(Existence and uniqueness)}} Given any $u_0\in O_\alpha$,  there is $t^+=t^+(u_0)\in(0,\infty]$ such that 
the semilinear  Cauchy problem \eqref{SCP} possesses a unique maximal strong  solution
\begin{equation} \label{reg:T0}
\begin{aligned}
 u(\cdot;u_0)&\in C^1\big((0, t^+),E_0\big)\cap C\big((0,t^+),E_1\big)\cap  C \big([0,t^+),O_\alpha\big) 
\end{aligned}
\end{equation}
 with
$$
\lim_{t\to 0}t^{\xi-\alpha}\|u(t;u_0)\|_\xi=0\,.
$$

{\bf \item[(ii)] {\em (Continuous dependence)}} The map $(t,u^0)\mapsto u(t;u_0)$ is a semiflow on $O_\alpha$. \vspace{1mm}

{\bf \item[(iii)] {\em (Global existence)}} If  $u([0,t^+(u_0));u_0)$ is relatively compact in $O_\alpha$, then~${t^+(u_0)=\infty}$.\vspace{1mm} 

{\bf \item[(iv)] {\em (Blow-up criteria)}} Let $u_0\in O_\alpha$ be such that $t^+(u_0)<\infty$. \vspace{1mm}
\begin{itemize}
\item[{\bf (a)}]  If  $u(\cdot;u_0):[0,t^+(u_0))\to E_\alpha$ is uniformly continuous, then
\begin{equation*} 
 \lim_{t\nearrow t^+(u_0)}\mathrm{dist}_{E_\alpha}\big( u(t;u_0),\partial O_\alpha\big)=0\,.
\end{equation*}

 \item[{\bf (b)}] If
$$
\underset{t\nearrow t^+(u_0)}{\limsup}\|f(u(t;u_0))\|_{0}<\infty\,,
$$
then
\begin{equation*} 
 \lim_{t\nearrow t^+(u_0)}\mathrm{dist}_{E_\alpha}\big( u( t;u_0),\partial  O_\alpha\big)=0\,.
\end{equation*}
\end{itemize}

\end{itemize}
\end{thm}

 {\tblue Lastly, let us assume, additionally to \eqref{Vor}, that
 \begin{subequations}\label{stas}
 \begin{equation}\label{stas1}
 0\in O_\alpha  
 \end{equation}
and that  there exist   constants~$c_0>0$, $\delta\in(0,1)$,  and $q_*>1$ with $q_*(\xi-\alpha)\leq 1+\gamma-\alpha$ such that 
\begin{equation}\label{stas2}
\|f(v)\|_{\gamma}\le   c_0 \|v\|_{\xi}^{q_*},\, \qquad v\in  O_\xi \cap\bar{\mathbb{B}}_{E_\alpha}(0,\delta)\,.
\end{equation}
 \end{subequations}
In particular, \eqref{stas}  ensures that $0$ is a stationary solution to \eqref{SCP}.
In Corollary~\ref{C1} we establish, by exploiting the   superlinear behavior of $f$ near zero in \eqref{stas2},
  the  exponential stability  of this equilibrium in the case when the} spectral bound~${s(A) := \sup\{\re\lambda: \lambda\in\sigma(A)\}}$ of $A$ is negative.

\begin{cor}\label{C1}
 Suppose~\eqref{Vor} and \eqref{stas}, assume that $s(A)<0$, and choose $\varpi\in(0,-s(A))$.
Then, there exists  an open neighborhood~$U_\alpha$  of $0$ in $E_\alpha$ and a constant $M\geq 1$ such that for each~${u_0\in U_\alpha}$ the maximal solution $u(\cdot;u_0)$ to \eqref{SCP} is globally defined, 
that is,  $t^+(u_0)=\infty$, and
$$
 \|u(t;u_0)\|_\alpha+t^{\xi-\alpha}\|u(t;u_0)\|_\xi\le M\, e^{-\varpi t}\,\|u_0\|_\alpha \,,\qquad t>0\,,\quad u_0\in U_\alpha\,.
$$
\end{cor}

The proof of Theorem~\ref{T:0x} is presented in Section~\ref{Sec2}. 
 The results of Theorem~\ref{T:2} and Corollary~\ref{C1} for the semilinear case  are proven  in Section~\ref{Sec3}.
Finally, in Section~\ref{Sec:4} we provide some applications of  these results  including {\tred a coagulation-fragmentation equation with size diffusion in an $L_1$-framework,} an asymptotic model for  atmospheric flows and    
   certain semilinear and quasilinear heat equations  featuring scaling invariances. 

\section{ The Quasilinear Problem: Proof of Theorem~\ref{T:0x}}\label{Sec2}

The proof of Theorem~\ref{T:0x} is based on Proposition~\ref{P:0} below.
Before stating this result, we   recall that given a Banach space $E$,  $\mu\in\R$, and $T>0$, the time weighted space
\[
 C_\mu\big((0,T],E\big):=\big\{u\in C((0,T], E)\,:\, \text{$t^\mu  \|u(t)\|_E\to 0$ for $t\to0$}\big\} 
 \]
 is a Banach space with the norm
\[
\|u\|_{C_\mu((0,T],E)}:=\sup_{t\in(0,T]}t^\mu\|u(t)\|_E\,. 
\]
For a suitable choice {\tblue of   $E$, this space}  plays a crucial role in the subsequent analysis.

\begin{prop}\label{P:0}
Assume \eqref{ASS} and let $\mu:=\xi-\alpha$. 
Then, given $\ov u\in O_\alpha,$ there exist constants~$r,T\in(0,1)$  such that for each $u_0\in\bar{\mathbb{B}}_{E_\alpha}(\ov u,r)\subset O_\alpha$ the Cauchy problem \eqref{QCP}  possesses a  strong  solution
\begin{equation}\label{regs'}
\begin{aligned}
u&\in C\big((0,T],E_1\big)\cap C^1\big((0,T],E_0\big)\cap C\big([0,T],O_\alpha\big)\\
&\quad \,\cap C_\mu\big((0,T],E_\xi\big)
 \cap  C^{\alpha-\beta}\big([0,T],E_\beta\big)\,.
\end{aligned}
\end{equation}
 Moreover, there exists a constant $c=c(\ov u)>0$ such that 
\begin{equation}\label{eq:14}
 \|u(t;u_0)-u(t;  u_1)\|_\alpha\leq c\|u_0- u_1\|_\alpha\,,\qquad u_0,\, u_1\in \bar{\mathbb{B}}_{E_\alpha}(\ov u,r)\,,\quad 0\leq t\leq T\,.
\end{equation}
 Finally, if 
\begin{equation}\label{regs''}
\wt u\in C\big((0,T],E_1\big)\cap C^1\big((0,T],E_0\big)\cap C_\mu\big((0,T],E_\xi\big) \cap  C^{\vartheta}\big([0,T],O_\beta\big)
\end{equation}
with $\vartheta\in(0,1)$ is a solution  to \eqref{QCP}, then $\wt u=u(\cdot; u_0)$.
\end{prop}

\begin{proof} We divide the proof into several steps.\vspace{2mm}

\noindent{\bf (i) Preliminaries.} 
Since $O_\beta$ is open in $E_\beta$,  there is $R\in(0,1)$ such that $\bar{\mathbb{B}}_{E_\beta}(\ov u,2R)\subset O_\beta$.
Moreover, recalling~\eqref{AS2},  we may assume that 
\begin{equation}\label{qe:1}
\|A(x) - A(y)\|_{\mathcal{L}(E_1,E_0)} \le K\|x-y\|_\beta\ , \quad x,\,y \in \bar{\mathbb{B}}_{E_\beta}(\ov u,2R)\, , 
\end{equation}
for some $K>0$
and, using~\cite[I.Theorem~1.3.1]{LQPP}, that  there are~${\kappa\ge 1}$ and $\omega>0$ such that
\begin{equation}\label{qe:2}
A(x)\in \mathcal{H}(E_1,E_0;\kappa,\omega)\,,\quad x\in \bar{\mathbb{B}}_{E_\beta}(\ov u,2R)\,.
\end{equation}

Fix $\rho\in (0,\alpha-\beta)$.
Given $T \in (0,1)$, we define a closed subset of $C([0,T], E_\beta)$ by
$$
\mathcal{V}_{T,R} := \left\{v \in C\big([0,T], \bar{\mathbb{B}}_{E_\beta}(\ov u,2R)\big) \,:\, \|v(t) - v(s)\|_\beta \le |t-s|^{\rho}\,, ~0 \le s, t \le T\right\}\,.
$$
Clearly, $\ov u\in \mathcal{V}_{T,R}$. Moreover, if $v\in\mathcal{V}_{T,R}$, then \eqref{qe:2} ensures 
\begin{subequations}\label{QE}
\begin{equation}\label{QE1}
A(v(t))\in \mathcal{H}(E_1,E_0;\kappa,\omega)\,,\quad t\in [0,T]\,,
\end{equation}
while \eqref{qe:1} implies
\begin{equation}\label{QE2}
A(v)\in C^\rho\big([0,T],\mathcal{L}(E_1,E_0)\big)\quad\text{and}\quad \sup_{0\le s<t\le T}
\frac{\|  A(v(t))-  A(v(s))\|_{\mathcal{L}(E_1,E_0)}}{( t-s)^\rho}\le K\,.
\end{equation}
\end{subequations} 
Therefore, for each $v\in\mathcal{V}_{T,R}$, the evolution operator
$$
U_{A(v)}(t,s)\,,\quad 0\le s\le t\le T\,,
$$
 for $A(v)$ is well-defined due to~\cite[II.Corollary~4.4.2]{LQPP}, and \eqref{QE} enables us to use the results of \cite[II.Section~5]{LQPP} (see assumption (5.0.1) therein).
In particular, we recall from \cite[II.Lemma~5.1.3]{LQPP} and~\eqref{Vor3q} that there exists a constant $\omega_0\ge 1$ such that
\begin{equation}\label{qe:a}
\|U_{A(v)}(t,s)\|_{\mathcal{L}(E_\theta)}+(t-s)^{\theta-\vartheta}\,\|U_{A(v)}(t,s)\|_{\mathcal{L}(E_\vartheta,E_\theta)}  \le \omega_0  \,, \quad 0\le s\le t\le T\,,
\end{equation}
for $\theta,\vartheta\in \{0,\gamma,\beta,\alpha,\xi,1\}$ with $\vartheta\le \theta$ and   $(\vartheta,\theta)\not\in\{(\beta,\alpha)\,, (\beta,\xi)\}$.  
Moreover, we  deduce  from~\eqref{qe:1}, \eqref{QE}, and \cite[II.Lemma~5.1.4]{LQPP}  that there is  
$\omega_1\ge 1$ such that, for $u,v\in \mathcal{V}_{T,R}$,
\begin{equation}\label{e23}
(t-\tau)^{\theta-\vartheta}\|U_{A(u)}(t,\tau)-U_{A(v)}(t,\tau)\|_{\mathcal{L}(E_\vartheta,E_\theta)}\le \omega_1\,\|u-v\|_{C([0,T],E_\beta)}\,,\quad 0\le \tau<t\le T\,,
\end{equation}
for  $\theta,\,\vartheta\in \{\gamma,\beta,\alpha,\xi\}$.  
We also note from \cite[II.Theorem~5.3.1]{LQPP} (with $f=0$ therein) that there is $\omega_2\ge 1$ with
 \begin{align}\label{p2}
\|U_{A(u)}(t,0)-  U_{A(u)}(s,0)\|_{\mathcal{L}(E_\alpha,E_\beta)}\leq \omega_2 (t-s)^{\alpha-\beta}\,,\quad 0\le s< t\le T\,,  
\end{align}
while, due to \cite[II.Equation~(5.3.8)]{LQPP} (with $r=s$ therein), we may also assume that
 \begin{equation}\label{p1}
\|U_{A(u)}(t,s)-1\|_{\mathcal{L}(E_\alpha,E_\beta)}\le \omega_2(t-s)^{\alpha-\beta}\,,\quad 0\le s< t\le T\,,
\end{equation}
for $u\in \mathcal{V}_{T,R}$.

We next observe that 
\begin{equation}\label{n}
U_{A(u)}(\cdot,0)u_0\in C_\mu((0,T],E_\xi)\,,\qquad u_0\in E_\alpha\,,\quad u\in\mathcal{V}_{T,R}\,.
\end{equation}
 Indeed,  we clearly have $U_{A(u)}(\cdot,0)u_0\in C((0,T],E_\xi)$. Moreover, since~\eqref{qe:a} ensures
\begin{align*}
t^\mu\|U_{A(u)}(t,0)u_0\|_\xi&\leq t^\mu\|U_{A(u)}(t,0)(u_0-  u_*)\|_\xi+t^\mu\|U_{A(u)}(t,0)  u_*\|_\xi\\[1ex]
&\leq \omega_0\| u_0-  u_*\|_\alpha+\omega_0 t^\mu\| u_*\|_\xi
\end{align*}
for $0< t\leq T$ and any $ u_*\in E_\xi$, we may use the density of $E_\xi$ in $E_\alpha$ and the premise $\mu>0$ to deduce
that $t^\mu\|U_{A(u)}(t,0)u_0\|_\xi\to 0$ as $t\to 0$, hence $U_{A(u)}(\cdot,0)u_0\in C_\mu((0,T],E_\xi)$ with
$$
\|U_{A(u)}(\cdot,0)u_0\|_{ C_\mu((0,T],E_\xi)}\le \omega_0\|u_0\|_\alpha\,.
$$
In particular, 
\begin{equation}\label{M}
M(T):=\|U_{A(\ov u)}(\cdot,0)\ov u\|_{C_\mu((0,T],E_\xi)}\to 0\qquad \text{as  $T\to 0$\,.}
\end{equation}

It is worth emphasizing that the above observations   still  hold when making $R$ or $T$ smaller. 
Let~${N_*:=N\big(\|\ov u\|_\beta+2R\big)>0}$ be the constant from~\eqref{AS4}.
Using the density of $E_\xi$ in $E_\beta$, we fix $\widetilde{u}\in E_\xi$ with 
{\tblue  \begin{equation*}
 \|\widetilde{u}-\ov u\|_\beta \le 2R \,.
\end{equation*}
Since $\wt u\in\bar{\mathbb{B}}_{E_\beta}\big(\ov u,2R\big)$,   we infer from~\eqref{AS4} (with~$w=\wt u$) that there exists a constant $K_*$ only depending  on $N_*$, $\|\wt u\|_\xi$, and~$\|f(\wt u)\|_\gamma$, such that for all $v,\, w\in E_\xi\cap \bar{\mathbb{B}}_{E_\beta}\big(\ov u,2R\big)$
 \begin{align}\label{estfw}
 \|f(v)\|_\gamma\leq \frac{K_*}{2}(1+ \|v\|_\xi^q)\,,
 \end{align}
 and
 \begin{align}
 \|f(w)-f(v)\|_\gamma&\leq \frac{K_*}{2} \big[\max\{\|v\|_\xi,\,\|w\|_\xi\}^{q}\|w-v\|_\beta+\max\{\|v\|_\xi,\,\|w\|_\xi\}^{q-1}\|w-v\|_\xi\big]\nonumber\\[1ex]
 & \quad+K_* \|w-v\|_\xi\,.\label{estfww}
 \end{align}
}
Since $q>1$, we find~${L\in (0,1)}$ such that
\begin{subequations}\label{choice}
\begin{align}
  (2\omega_0+\omega_1){\tblue  K_*} (\mathsf{B}_\beta +\mathsf{B}_\xi )L^{q-1}\le\frac{1}{4}\,,\label{512a}
\end{align}
 where 
\begin{equation}\label{defBtheta}
\mathsf{B}_\theta:=\mathsf{B}(1+\gamma-\theta,1-\mu q)\,,\quad \theta\in[0,1]\,,
\end{equation}
with $\mathsf{B}$ denoting the Beta function.
We may then choose $r\in (0,R)$ such that
 \begin{equation}
r(\omega_0+2\omega_1 \|\ov u\|_\alpha) \le \frac{L}{4}\,,\qquad 4\omega_1 r\le\frac{1}{8}\,.\label{512b}
\end{equation}
Moreover, using the density of $E_\xi$ in $E_\alpha$, we fix $\widehat{u}\in E_\xi$ with 
 \begin{equation}
\|\ov u-\widehat{u}\|_\alpha\le r\label{512c}\,.
\end{equation}
Consequently, we can fix $T\in (0,1)$ such that (recall~\eqref{M}, $\rho\in (0,\alpha-\beta)$, $\mu>0$)
\begin{align}
&{\tblue T^\mu\leq L}\,,\qquad M(T)\le \frac{L}{4}\,,\qquad T^\rho\leq r\,,\qquad  \omega_1\|\hat{u}\|_\xi T^\mu\le \frac{1}{16}\,, \label{512d}
\end{align}
and 
\begin{align}
 \big[ \omega_2\big(\|\ov u\|_\alpha+r\big)+\omega_0\,{\tblue  K_*}\,  \big(\omega_2\mathsf{B}_\alpha+\mathsf{B}_\beta\big) \big ] T^{\alpha-\beta-\rho}\le 1\,.\label{512e}
\end{align}
\end{subequations}
Note that $T$ and $r$ only depend on $\ov u \in O_\alpha$.\medskip

\noindent{\bf (ii) The fixed point mapping.} Let $L\in (0,1)$, $r\in (0,R)$, and $T\in (0,1)$ be chosen as in~\eqref{choice}. 
Let $\mathsf{e}_{\alpha,\beta}$ be the norm of the continuous embedding of $E_\alpha$ in $E_\beta$ and assume without loss of generality that   $\mathsf{e}_{\alpha,\beta}\ge 1$. 
Fix 
$$
u_0\in\bar{\mathbb{B}}_{E_\alpha}\big(\ov u,r/\mathsf{e}_{\alpha,\beta}\big)\,.
$$ 
We then introduce the   complete metric space
$$
\mathcal{W}:=\mathcal{V}_{T,r}\cap  \bar{\mathbb{B}}_{C_\mu((0,T],E_\xi)}(0,L)
$$
equipped with the metric
$$
d_{\mathcal{W}}(u,v):=\|u-v\|_{C([0,T],E_\beta)}+\|u-v\|_{C_\mu((0,T],E_\xi)}\,,\qquad u,\,v\in \mathcal{W}\,,
$$
and define a  mapping $F$ on $\mathcal{W}$ by
\begin{equation}\label{Lambda}
F(u)(t):= U_{A(u)}(t,0)u_0 +\int_0^t U_{A(u)}(t,\tau) f(u(\tau))\,\rd \tau\,,\qquad t\in [0,T]\,,\quad u\in\mathcal{W}\,.
\end{equation}
 We shall show that $F:\mathcal{W}\rightarrow \mathcal{W}$ is in fact a contraction.
\medskip

\noindent{\bf (iii) Self-mapping.} We check that $F$ is a self-mapping on $\mathcal{W}$.
 {\tblue Given $u,\, v\in \mathcal{W}$, it follows from \eqref{estfw}-\eqref{estfww}, since $Lt^{-\mu}\geq 1$ for $t\in(0,T]$ by \eqref{512d}, that 
\begin{subequations}\label{f}
\begin{align}
\|f(u(t))\|_\gamma&\le \frac{K_*}{2}(1+\|u(t)\|_\xi^q) \le \frac{K_*}{2}\big(1+\|u\|_{C_\mu((0,t],E_\xi)}^qt^{-\mu q} \big) \label{217a}\\[1ex]
&\le  \frac{K_*}{2}\big(1+L^qt^{-\mu q} \big) \le  K_*L^qt^{-\mu q}  \label{217b}
\end{align}
\end{subequations}
 and
\begin{equation}\label{f'}
\begin{aligned}
\|f(u(t))-f(v(t))\|_\gamma&\le  K_*\big(t^{-\mu}+L^{q-1}t^{-\mu q}\big)\,d_{\mathcal{W}}(u,v)\\[1ex]
&\le  2K_* L^{q-1}t^{-\mu q} \,d_{\mathcal{W}}(u,v) 
\end{aligned}
\end{equation}
for $t\in  (0,T]$.}
We first show that $F(u)\in \mathcal{V}_{T,r}$. To this end, let  $0\le s<t\le T$.
Then
\begin{align}
\|F(u)(t)-F(u)(s)\|_\beta&\le \|U_{A(u)}(t,0)u_0-U_{A(u)}(s,0)u_0\|_{\beta} \nonumber\\
&\quad+\int_0^s \|U_{A(u)}(t,\tau)-U_{A(u)}(s,\tau)\|_{\mathcal{L}(E_\gamma,E_\beta)} \,\|{ f(u(\tau))}\|_\gamma\,\rd \tau \nonumber\\
&\quad 	+\int_s^t \|U_{A(u)}(t,\tau)\|_{\mathcal{L}(E_\gamma,E_\beta)} \,\|{ f(u(\tau))}\|_\gamma\,\rd \tau\nonumber\\
&=:I_1+I_2+I_3\,.\label{p3}
\end{align}
In view of \eqref{p2}  we have 
\begin{align}\label{qe:6}
I_1\leq \omega_2 \|u_0\|_\alpha (t-s)^{\alpha-\beta}\leq \omega_2 \big(\|\ov u\|_\alpha+r\big) (t-s)^{\alpha-\beta}\,.
\end{align} 
Next,  we use \eqref{AS1}, \eqref{qe:a}, \eqref{p1},  and~\eqref{f} to derive
\begin{align}
I_2 &\le \int_0^s \|U_{A(u)}(t,s)-1\|_{\mathcal{L}(E_\alpha,E_\beta)} \,\|U_{A(u)}(s,\tau)\|_{\mathcal{L}(E_\gamma,E_\alpha)} \|f(u(\tau))\|_\gamma\,\rd \tau\nonumber\\
&\le \omega_0 \, \omega_2\, K_*   L^q (t-s)^{\alpha-\beta}\, \int_0^s (s-\tau)^{\gamma-\alpha} \tau ^{-\mu q}\,\rd \tau \nonumber\\
&\le \omega_0 \, \omega_2\,K_*   L^q  \, \mathsf{B}(1+\gamma-\alpha,1-\mu q)\, (t-s)^{\alpha-\beta}\,. \label{p4}
\end{align}
Similarly, we obtain from \eqref{AS1},  \eqref{qe:a},  and~\eqref{217b} that
\begin{align}
I_3 &\le     \omega_0\, K_*  L^q \int_s^t (t-\tau)^{\gamma-\beta}  \,\tau^{-\mu q}\,\rd \tau \le    \omega_0\, K_* L^q \,\mathsf{B}(1+\gamma-\beta,1-\mu q)\, (t-s)^{\alpha-\beta} \,.\label{p5}
\end{align}
Gathering~\eqref{p3}-\eqref{p5} and recalling that  $L<1$ and $(t-s)^{\alpha-\beta}\le T^{\alpha-\beta-\rho}(t-s)^\rho$, it follows from~\eqref{512e}  that
\begin{align}\label{e12}
\|F(u)(t)-F(u)(s)\|_\beta &\le   (t-s)^{\rho}\,,\quad 0\le s\le t\le T\,.
\end{align}
In particular, since $F(u)(0)=u_0\in\bar{\mathbb{B}}_{E_\alpha}\big(\ov u,r/\mathsf{e}_{\alpha,\beta}\big)$, we  infer that
\begin{equation}\label{e19}
\|F(u)(t)-\ov u\|_\beta \le  \|F(u)(t)-F(u)(0)\|_\beta +  \|u_0-\ov u\|_\beta \le T^{\rho}+r \le 2r\,,\quad 0\le t\le T\,,
\end{equation}
 using~\eqref{512d}.
In view of \eqref{e12} and \eqref{e19} we thus have  $F(u)\in \mathcal{V}_{T,r}$.

Next, we show that 
\begin{equation}\label{bbv}
\sup_{t\in(0,T]}t^\mu\|F(u)(t)\|_\xi\leq L\qquad\text{and}\qquad \lim_{t\to 0}t^\mu\|F(u)(t)\|_\xi= 0\,.
\end{equation}
 To this end, we observe first from \eqref{e23},  \eqref{512b}, and \eqref{512d} that
 \begin{align}\label{bbv1}
t^\mu \| U_{A(u)}(t,0)u_0\|_\xi &\le t^\mu \| U_{A(u)}(t,0)(u_0-\ov u)\|_\xi+t^\mu \| (U_{A(u)}(t,0)-U_{A(\ov u)}(t,0))\ov u\|_\xi\nonumber\\[1ex]
&\quad+t^\mu \| U_{A(\ov u)}(t,0)\ov u\|_\xi \nonumber\\
&\le r(\omega_0+2\omega_1\|\ov u\|_\alpha)+M(T)\leq L/2
\end{align}
 for $t\in (0,T]$. 
 Next, using  ~\eqref{AS1},~\eqref{qe:a}, and~\eqref{217a} we obtain {\tblue for $ t\in (0,T]$, 
 \begin{equation}\label{bbv2ad}
\begin{aligned}
t^\mu \|F(u)(t)- U_{A(u)}(t,0)u_0\|_\xi &\le t^\mu \int_0^t \|U_{A(u)}(t,\tau)\|_{\mathcal{L}(E_\gamma,E_\xi)} \,\|f(u(\tau))\|_\gamma\,\rd \tau\\
&\le  \frac{\omega_0 K_*}{2}t^\mu\int_0^t  (t-\tau)^{\gamma-\xi}\,\rd \tau\\
&\quad +  \frac{\omega_0 K_*}{2}\|u\|_{C_\mu((0,t],E_\xi)}^qt^\mu\int_0^t  (t-\tau)^{\gamma-\xi}\tau^{-\mu q}\,\rd \tau\\
&= \frac{\omega_0 K_*}{2(\xi-\gamma)}t^{1+\mu+\gamma-\xi}+ \frac{\omega_0 K_*\mathsf{B}_\xi}{2}\, \|u\|_{C_\mu((0,t],E_\xi)}^q\,,
\end{aligned}
\end{equation}
and, analogously, using ~\eqref{AS1}, \eqref{512a}, and \eqref{217b},
 \begin{equation}\label{bbv2}
\begin{aligned}
t^\mu \|F(u)(t)- U_{A(u)}(t,0)u_0\|_\xi &\le t^\mu \int_0^t \|U_{A(u)}(t,\tau)\|_{\mathcal{L}(E_\gamma,E_\xi)} \,\|f(u(\tau))\|_\gamma\,\rd \tau\\
&\le \omega_0\,K_*\, t^{1+\gamma-\xi-\mu q+\mu}\,\mathsf{B}_\xi\, L^q\\
& =\omega_0\,K_*\mathsf{B}_\xi L^q\leq L/2\,.
\end{aligned}
\end{equation}
Due to~\eqref{AS1} and $u\in C_\mu((0,T],E_\xi)$,  the left-hand side  of \eqref{bbv2ad} goes to zero as $t\to 0$ .}
It now follows from the relations \eqref{n}, \eqref{bbv1}, and \eqref{bbv2} that the properties stated  in \eqref{bbv} are both satisfied.

Recalling $F(u)\in \mathcal{V}_{T,r}$ and \eqref{bbv},  it remains to prove the continuity  $F(u)\in C((0,T],E_\xi)$  to conclude that $F(u)\in \mathcal{W}$.
We show  in fact  the stronger property \mbox{$F(u)\in C((0,T], E_1)$}.
To this end we fix $\e\in (0,T)$ and set $u_\e(t):=u(t+\e)$ for $t\in [0,T-\e]$. 
Then,   $u_\e\in C([0,T-\e],E_\xi)$ and  we may infer from \eqref{AS4} that  $f(u_\e)\in C([0,T-\e],E_\gamma)$.
Note that
 \[
 U_{A(u_\varepsilon)}(t,s)=U_{A(u)}(t+\varepsilon,s+\varepsilon)\,,\quad 0\leq s\leq t\leq T-\varepsilon\,,
\]
 is the evolution operator for $A(u_\e)$. 
 Using the definition \eqref{Lambda} of $F(u)$ we  get
\begin{equation}\label{P1d}
F(u)(t+\e)=U_{A(u_\e)}(t,0)F(u)(\e)+\int_0^t U_{A(u_\e)}(t,s) f(u_\e(s))\,\rd s\,,\quad t\in [0,T-\e]\,,
\end{equation}
and   \cite[II.Theorem~1.2.2, II.Remarks~2.1.2 (e)]{LQPP} combined ensure  now that  $$F(u)( \e+\cdot) \in  C\big((0,T-\varepsilon],E_1\big)\cap C^1\big((0,T-\e],E_0\big).$$ 
Since $\e\in(0,T)$ is arbitrary, we conclude that indeed 
\begin{equation}\label{regdes}
F(u)\in C\big((0,T],E_1\big)\cap C^1\big((0,T],E_0\big)\,,
\end{equation}
hence $F:\mathcal{W}\rightarrow \mathcal{W}$ is a well-defined mapping.
\medskip

\noindent{\bf (iv) Contraction Property.} Finally, we show that $F:\mathcal{W}\rightarrow \mathcal{W}$ is contractive, that is,
\begin{align}\label{contr}
d_{\mathcal{W}}\big(F(u),F(v)\big) &\le  \frac{1}{2}\,d_{\mathcal{W}}(u,v) \,,\quad u,\,v\in \mathcal{W}\,.
\end{align}
Consider $u,v\in \mathcal{W}$ and $t\in (0,T]$.
Given $\theta\in \{\beta,\alpha,\xi\}$, we first note from   \eqref{e23} and \eqref{512c} that
\begin{align}\label{c110}
\|U_{A(u)}(t,0)u_0-U_{A(v)}(t,0)u_0\|_{\theta}&\le \big\|U_{A(u)}(t,0)-U_{A(v)}(t,0)\|_{\kL(E_\alpha,E_\theta)}\|u_0-\wh u\big\|_{\alpha}\nonumber\\
&\quad+\big\|U_{A(u)}(t,0)-U_{A(v)}(t,0)\|_{\kL(E_\xi,E_\theta)}\|\wh u\|_{\xi}\nonumber\\
&\le \omega_1\big( 2r t^{\alpha-\theta}  + t^{\xi-\theta}\|\wh u \|_\xi\big)\, \|u-v\|_{C([0,T],E_\beta)}\,.
\end{align}

 Taking successively $\theta=\beta$  and $\theta=\xi$ in~\eqref{c110}, we infer from~\eqref{512b} and~\eqref{512d} that
\begin{align}\label{d111}
d_{\mathcal{W}}\big(U_{A(u)}(\cdot,0)u_0 , U_{A(v)}(\cdot,0)u_0\big)\le\omega_1\big( 4r   + 2T^{\mu}\|\wh u \|_\xi\big)d_{\mathcal{W}}(u,v)\leq \frac{1}{4} d_{\mathcal{W}}(u,v)\,.
\end{align}
Set
$$
\Phi(u):=F(u)-U_{A(u)}(\cdot,0)u_0
$$ 
and note from    \eqref{qe:a},  \eqref{e23}, \eqref{f}, and \eqref{f'} that, for $\theta\in \{\beta,\alpha,\xi\}$ and $t\in (0,T]$,
\begin{equation}\label{c111}
\begin{aligned}
\|\Phi(u)(t)-\Phi(v)(t)\|_\theta &\le \int_0^t \|U_{A(u)}(t,\tau)-U_{A(v)}(t,\tau)\|_{\mathcal{L}(E_\gamma,E_\theta)} \,\|f(u(\tau))\|_\gamma\,\rd \tau \\
&\quad 	+\int_0^t \|U_{A(v)}(t,\tau)\|_{\mathcal{L}(E_\gamma,E_\theta)} \,\|f(u(\tau))-f(v(\tau))\|_\gamma\,\rd \tau \\
&\le   \big(2\omega_0 +\omega_1\big)\,\mathsf{B}_\theta\, {\tblue K_*}  L ^{q-1}\,t^{1+\gamma-\theta-\mu q}\,d_{\mathcal{W}}(u,v)  \,.
\end{aligned}
\end{equation}
Taking $\theta=\beta$ and $\theta=\xi$ in \eqref{c111}, we obtain together with \eqref{512a} that
\begin{align*}
d_{\mathcal{W}}\big(\Phi(u),\Phi(v)\big) &\le  \big(2\omega_0 +\omega_1\big)\, \, K_* (B_\beta+B_\xi) L ^{q-1}\,d_{\mathcal{W}}(u,v)  \leq \frac{1}{4}\,d_{\mathcal{W}}(u,v) \,,
\end{align*}
and combining this with~\eqref{d111} we  arrive at the desired  property~\eqref{contr}.

Consequently, $F:\mathcal{W}\to \mathcal{W}$ is a contraction  so that Banach's fixed point theorem implies that   $F$ has 
a unique fixed point $u=u(\cdot;u_0)\in \mathcal{W}$. \\

 \noindent{\bf (v) Regularity Properties.} As shown  when deriving  \eqref{regdes}, the fixed point  $u$    of $F$  belongs to $ C((0,T],E_1)\cap C^1((0,T],E_0)$ and solves~\eqref{QCP}.
Moreover, $u(t)\to u_0$ in $E_\alpha$  as~${t\to0}$. 
Indeed, since $u_0\in E_\alpha$, we have $U_{A(u)}(t,0)u_0\to u_0$ in $E_\alpha$ as $t\to0$, while,  recalling \eqref{AS1},~\eqref{qe:a}, and~\eqref{217a}, 
\begin{align*}
\|\Phi(u)(t)\|_\alpha&\le \int_0^t \|U_{A(u)}(t,\tau)\|_{\mathcal{L}(E_\gamma,E_\alpha)} \,\|f(u(\tau))\|_\gamma\,\rd \tau \\[1ex]
& \le {\tblue \frac{\omega_0K_*}{2(\alpha-\gamma)} t^{1+\gamma-\alpha} +\frac{\omega_0K_*\mathsf{B}_\alpha}{2}\|u\|_{C_\mu((0,t],E_\xi)}^q\,,\qquad t\in(0,T]\,,}
\end{align*}
and the right-hand side tends to zero as $t\to 0$,
hence  $u\in C\big([0,T],E_\alpha\big)$.  It remains to observe from \eqref{p3}-\eqref{p5}
that $u\in C^{\alpha-\beta}\big([0,T],E_\beta\big)$ to get the regularity property~\eqref{regs'}.\medskip

\noindent{\bf (vi) Uniqueness.}   Let $u_0\in\bar{\mathbb{B}}_{E_\alpha}(\ov u,r)$   and~$u=u(\cdot;u^0)$. Let $\tilde u$ solve \eqref{QCP} with initial datum~$u_0$.
Choosing $\wt \rho\in\big(0,\min\{\rho,\, \vartheta\}\big)$, we note  that, if $T$ is sufficiently small,
 then both functions $u$ and  $\tilde u$ belong to the complete metric space $\mathcal{W}$ (with $\rho$ replaced by $\wt \rho$).
The uniqueness claim is now a straightforward consequence of the contraction property of $F$.\medskip

 \noindent{\bf (vii) Continuous Dependence.}
Given $u_0,\, u_1\in \bar{\mathbb{B}}_{E_\alpha}(\ov u,r)$ let $u(\cdot;u_0)$ and $u(\cdot;u_1)$ be the corresponding solutions derived above.
 Denoting still by~$F$ the mapping introduced in~\eqref{Lambda} (associated with~$u_0$) we have 
\begin{align}\label{d0}
 u(\cdot;u_0)=F(u(\cdot;u_0))\,,\qquad u(\cdot;u_1)=U_{A(u(\cdot;u_1))}(\cdot,0)(u_1-u_0)+ F(u(\cdot;u_1))\,.
\end{align}
Therefore, using~\eqref{qe:a} and~\eqref{contr} we  deduce that
\begin{align*}
d_{\mathcal{W}}\big(u(\cdot;u_0),u(\cdot;u_1)\big) &\le d_{\mathcal{W}}\big(F(u(\cdot;u_0)),F(u(\cdot;u_1))\big)\\ &\quad +d_{\mathcal{W}}\big(U_{A(u(\cdot;u_1))}(\cdot,0)u_0,U_{A(u(\cdot;u_1))}(\cdot,0)u_1\big)\\
&\le \frac{1}{2}\, d_{\mathcal{W}}\big(u(\cdot;u_0),u(\cdot;u_1)\big)+\omega_0\, (\mathsf{e}_{\alpha,\beta}+1)\,\|u_0-u_1\|_\alpha\,,
\end{align*}
where we still use $\mathsf{e}_{\alpha,\beta}$ for the norm of the embedding $E_\alpha\hookrightarrow E_\beta$. Consequently,
\begin{align}\label{d1}
d_{\mathcal{W}}\big(u(\cdot;u_0),u(\cdot;u_1)\big) &\le 2\,\omega_0\, (\mathsf{e}_{\alpha,\beta}+1)\,\|u_0-u_1\|_\alpha\,.
\end{align}
Now, recall from~\eqref{e23} and~\eqref{c111} that, for $t\in [0,T]$,
\begin{align*}
\big\|F(u(\cdot;&u_0))(t)-F(u(\cdot;u_1))(t)\big\|_\alpha\\
 &\le   \big[\omega_1 \big(r+\|\ov u\|_\alpha\big)+\big(2\omega_0 +\omega_1\big)\mathsf{B}_\alpha K_*  L ^{q-1}\big]\, d_{\mathcal{W}}\big(u(\cdot;u_0),u(\cdot;u_1)\big)  \,, 
\end{align*}
hence \eqref{d1} entails that
\begin{align*}
\big\|F(u(\cdot;u_0))(t)-F(u(\cdot;u_1))(t)\big\|_\alpha &\le  c_2  \,\|u_0-u_1\|_\alpha\,,\quad t\in [0,T]\,,
\end{align*}
for some constant $c_2>0$. 
Invoking~\eqref{qe:a} and~\eqref{d0},  we  finally obtain  for $t\in [0,T]$ that
\begin{align*}
\big\|u(\cdot;u_0)(t)-u(\cdot;u_1)(t)\big\|_\alpha &\le  \big\|U_{A(u(\cdot;u_1))}(\cdot,0)(u_1-u_0)\big\|_\alpha\\
&\quad +\big\|F(u(\cdot;u_0))(t)-F(u(\cdot;u_1))(t)\big\|_\alpha\\
&\le c_3 \,\|u_0-u_1\|_\alpha
\end{align*}
for some constant $c_3>0$. This guarantees~\eqref{eq:14} and completes the proof of Proposition~\ref{P:0}.
\end{proof}

 \begin{rem}\label{Rem1}
The proof of Proposition~\ref{P:0} shows that
the H\"older continuity in time and the assumption~$\alpha>\beta$ are only needed to ensure~\eqref{QE2} while the assumptions~${\gamma>0}$  and~${\xi<1}$ are
  only used when applying formula~\eqref{e23} to derive (Lipschitz) continuity properties of the evolution operator $U_{A(u)}(t,s)$ with respect to $u$. 
   In other words, these assumptions are required to handle the quasilinear part and can thus be weakened for
   the semilinear  problem~\eqref{SCP}.
\end{rem}

\subsection*{Proof of Theorem~\ref{T:0x}}
We are now in a position to establish Theorem~\ref{T:0x}. \\

\noindent{\bf (i),\, (ii) Existence and Uniqueness.}
Proposition~\ref{P:0} ensures that  the Cauchy problem~\eqref{QCP} has for each  $u_0\in O_\alpha$  a unique local strong solution. 
This solution  can be extended by standard arguments to 
a maximal strong solution $u(\cdot;u_0)$ on  a maximal interval of existence.
 The regularity properties \eqref{reg:T0x}  and the uniqueness claim   follow from Proposition~\ref{P:0}.\medskip

\noindent{\bf (iii) Continuous dependence:} It suffices to prove that, given $u_0\in O_\alpha$ and  $t_*\in (0,t^+(u_0))$, there exist constants $\e>0$ and $K>0$ such that 
for all $\wt u_0\in  \mathbb{B}_{E_\alpha}(u_0, \e_0)\subset O_\alpha$, the maximal strong solution $u(\cdot; \wt u_0)$ is defined on $[0,t_*]$, that is, $t^+(\wt u_0)>t_*$, and satisfies
\begin{equation}\label{bubx}
\|u(t; u_0)- u(t; \wt u_0)\|_{ \alpha} \leq K\|u_0-\wt u_0\|_\alpha\,,\quad 0\leq t\leq t_*\,.
\end{equation}
Let thus $u_0\in O_\alpha$ and  fix  $t_*\in (0,t^+(u_0))$.
For each~${s \in[0,t_*]}$,  we infer  from Proposition~\ref{P:0} with  $\ov u:=u(s;u_0)$  that there exist constants $r_s,\,T_s\in(0,1)$ 
such that problem~\eqref{QCP} has  for each $y\in \bar{\mathbb{B}}_{E_\alpha} (u(s;u_0),r_s)$ a unique  strong solution~${u=u(\cdot;y)}$ on~${[0,T_s]}$.
Moreover, there exists a  constant $c_s>0$ such that 
 \begin{equation*}
 \|u(t;y_0)-u(t; y_1)\|_\alpha\leq c_s\|y_0-y_1\|_\alpha\,,\qquad y_0,\,y_1\in \bar{\mathbb{B}}_{E_\alpha}(u(s;u_0),r_s)\,,\quad 0\leq t\leq T_s\,,
\end{equation*}
see \eqref{eq:14}.
Since the set $u([0,t_*];u_0)\subset O_\alpha$ is compact, it possesses a finite cover
\[
 u([0,t_*];u_0)\subset \bigcup_{i=1}^n V_i\qquad\text{with}\qquad V_i:=\mathbb{B}_{E_\alpha}(u(s_i;u_0), r_{s_i})\subset O_\alpha\,.
\]
We define
\[
 T:=\min\{T_{s_i}\,:\, 1\leq i\leq n\},\qquad C:= \max\{c_{s_i}\,:\, 1\leq i\leq n\}\,.
\]
Since  $u_0\in V_{i_0}$ for some $1\leq i_0\leq n,$  there exits $\varepsilon>0$  with $\mathbb{B}_{E_\alpha}(u_0, \e)\subset V_{i_0}$.
Therefore, for $\wt u_0\in \mathbb{B}_{E_\alpha}(u_0, \e)$, the solution  $u(\,\cdot\,;\wt u_0)$ provided by Proposition~\ref{P:0} satisfies, according to~\eqref{eq:14},
\begin{equation}\label{bub1x}
\|u(t; u_0)- u(t; \wt u_0)\|_{ \alpha} \leq C\|u_0-\wt u_0\|_\alpha\,,\quad 0\leq t\leq T\,.
\end{equation}
 If $t_*\leq T$, we obtain \eqref{bubx}  and the claim follows.
If  however $T<t_*$, there is~${1\leq i_1\leq n}$ such that  $ u(T/2;u_0)\in V_{i_1}$. 
Moreover, making  $\e$ smaller (if necessary), we may assume in view of~\eqref{bub1x}  that   $u(T/2;\wt u_0)\in V_{i_1}$  for all~${\wt u_0\in \mathbb{B}_{E_\alpha}(u_0, \e)}$.
Hence, the solution $u(\cdot; u(T/2;\wt u_0))$ is also defined on~${[0,T]}$
and satisfies, due to \eqref{eq:14} and \eqref{bub1x},
\begin{align}\label{bub2x}
\|u(t; u(T/2; u_0))- u(t; u(T/2;\wt u_0))\|_{ \alpha}& \leq C\|u(T/2; u_0)-u(T/2;\wt u_0)\|_\alpha\nonumber\\
&\leq  C^2\|u_0-\wt u_0\|_\alpha
\end{align}
for $0\leq t\leq T$.
Since the uniqueness result in  Proposition~\ref{P:0} ensures that 
$$
u(t+T/2; \wt u_0)=u(t;u(T/2;\wt u_0))\,,\qquad t\in[0,T/2]\,,\quad \wt u_0\in \mathbb{B}_{E_\alpha}(u_0, \e)\,,
$$
the solution  $u(\cdot; \wt u_0)$ is defined (at least on) on $[0,3T/2]$ 
and satisfies   
\[
u(t;\wt u_0)
=\left\{
\begin{array}{lll}
u(t;\wt u_0)\,,& t\in[0,T]\,,\\[1ex]
u(t-T/2; u(T/2;\wt u_0))\,, &t\in[T/2,3T/2]\,,
\end{array}
\right.
\]
along with 
\begin{equation*}
\|u(t; u_0)- u(t; \wt u_0)\|_{ \alpha} \leq C^2\|u_0-\wt u_0\|_\alpha\,,\quad 0\leq t\leq 3T/2\,,
\end{equation*}
 according to \eqref{bub1x}-\eqref{bub2x}.

If $t_*\leq  3T/2$, then the proof is complete. 
Otherwise, we proceed as above to derive~\eqref{bubx} after a finite number of steps. \medskip

\noindent{\bf (iv)  Global existence:} The semiflow property of the solution map  ensures for  relatively compact  orbits $u([0,t^+(u_0));u_0)$ in $O_\alpha$ that  $t^+(u_0)=\infty$.\medskip

\noindent{\bf (v)  Blow-up criteria:} Let $u_0\in O_\alpha$ be such that $t^+:=t^+(u_0)<\infty$.\vspace{1mm}

\indent {\bf (a)} If $u(\cdot;u_0):[0,t^+)\to E_\alpha$ is uniformly continuous and \eqref{boundx} does not hold, we deduce that  the limit $\lim_{t\nearrow t^+}u(t;u_0)$ exists in~$O_\alpha$.
Consequently,~$u([0,t^+);u_0)$ is relatively compact in~$O_\alpha$, which contradicts~{\bf (iv)}. 
\vspace{2mm}

\indent {\bf  (b)} If $u(\cdot;u_0):[0,t^+)\to E_\beta$ is uniformly H\"older continuous and~\eqref{v1} does not hold,  we have
{\tblue \begin{equation}\label{ff}
\sup_{t\in(t^+/2,t^+)}\|f(u(t;u_0))\|_0<\infty\,,
\end{equation}
and} we may extend  $u(\cdot;u_0)$ to a function in $C^\rho([0,t^+], O_\beta)$ for some $\rho\in(0,1)$,  hence~\eqref{AS2} implies that $A(u(\cdot;u_0))\in C^\rho\big([0,t^+],\mathcal{H}(E_1,E_0)\big)$.
Thus, we infer from~\cite[II.Corollary~4.4.2]{LQPP} that the  evolution operator
$$
U_{A(u(\cdot;u_0))}(t,s)\,,\quad 0\le s\le t\le t^+\,,
$$
for $A(u(\cdot;u_0))$ is well-defined and satisfies~\cite[II.Assumption~(5.0.1)]{LQPP} (see~\cite[I.Corollary~1.3.2]{LQPP}), so that the corresponding stability estimates~\eqref{qe:a} are valid on~$[0,t^+]$. 
{\tblue Based on the fact that~${u\in C_\mu((0,t^+/2], E_\xi)}$, the assumption  \eqref{AS4} ensures that there exists a constant $c>0$ with
\begin{align}\label{fff}
\|f(u(t;u_0))\|_\gamma\le c\, t^{-\mu q}\,,\quad 0<t\le t^+/2\,,
\end{align} 
and} it now follows from~\eqref{qe:a},~\eqref{ff},~\eqref{fff}, and the variation-of-constants formula~\eqref{Lambda} that the limit
\begin{equation*}
\lim_{t\nearrow t^+} u(t;u_0)= U_{A(u(\cdot;u_0))}(t^+,0)u_0 +\int_0^{t^+} U_{A(u(\cdot;u_0))}(t^+,\tau) f(u(\tau))\,\rd \tau
\end{equation*}
exists in~$E_\alpha$.  Consequently, the orbit $u([0,t^+);u_0)$ is relatively compact in $O_\alpha$, hence $t^+=\infty$ by~{\bf (iv)}. This is a contradiction.\vspace{2mm}

\indent {\bf (c)} Let $E_1$ be compactly embedded in~$E_0$. Assume~\eqref{v2} is not true for some~\mbox{$\theta\in (\alpha,1)$}. 
Hence, for any sequence $t_n\nearrow t^+$ there  exist a subsequence $(t_{n_j})$ of $(t_n)$ and $\ov u\in O_\alpha$ such that~${u(t_{n_j};u_0)\to \ov u}$ in~$E_\alpha$.  
Therefore, the orbit  $u([0,t^+);u_0)$ is relatively compact in $O_\alpha$, which is a contradiction.

\qed

\section{The Semilinear Case: Proof of Theorem~\ref{T:2}  and Corollary~\ref{C1}}\label{Sec3}

We briefly  indicate how to prove Theorem~\ref{T:2}, which  relies on the following result.
\begin{prop}\label{P:1}
 Assume \eqref{Vor}. 
Then, given $\ov u\in O_\alpha$, there exist constants~$r,\,T\in(0,1)$  such that for all $u_0\in\bar B_{E_\alpha}(\ov u,r)$ there  exists a unique strong  solution $u=u(\cdot;u_0)$ to \eqref{SCP} such that
\begin{equation}\label{regso}
u\in C\big((0,T],E_1\big)\cap C^1\big((0,T],E_0\big)\cap C\big([0,T],O_\alpha\big)\cap C_{\xi-\alpha}\big((0,T],E_\xi\big)\,.
\end{equation}

 Moreover, there exists a constant $c>0$ such that 
\begin{equation}\label{eq:maybe}
 \|u(t;u_0)-u(t; u_1)\|_\alpha\leq c\|u_0-u_1\|_\alpha\,,\qquad u_0,\,u_1\in \bar B_{E_\alpha}(\ov u,r)\,,\quad 0\leq t\leq T\,.
\end{equation}
\end{prop}

\begin{proof}
Letting  $(e^{tA})_{t\geq0}$ be the strongly continuous analytic semigroup on $E_0$ generated by~$A$, we formulate  \eqref{SCP} as a fixed point problem $u=F(u)$, where
\begin{equation}\label{FFF}
F(u)(t):= e^{tA} u_0+ \int_0^t e^{(t-\tau)A} f(u(\tau))\,\rd \tau\,,\qquad t\in [0,T]\,.
\end{equation} 
For suitable $r,\, L,\, T\in(0,1)$ and $u_0\in \bar{\mathbb{B}}_{E_\alpha}(\ov u,r)$,  
one proves analogously to the proof of Proposition~\ref{P:0} that  $F:\mathcal{W}\to\mathcal{W}$ is a contraction for
\[
 \mathcal{W}=\bar{\mathbb{B}}_{C([0,T],E_\alpha)}(\ov u,L)\cap \bar{\mathbb{B}}_{C_\mu((0,T],E_\xi)}(0,L)\,,
\]
where we refer to Remark~\ref{Rem1} for the weaker assumptions required in the semilinear setting.
\end{proof}

Based on Proposition~\ref{P:1}, we now establish  Theorem~\ref{T:2}.

\subsection*{Proof of Theorem~\ref{T:2}}
The proof of Theorem~\ref{T:2} follows from Proposition~\ref{P:1}  by arguing as in the proof  of Theorem~\ref{T:0x}.
\qed
\medskip

 We conclude this section with the proof of Corollary~\ref{C1}.

\subsection*{Proof of Corollary~\ref{C1}}

 {\tblue  Letting $c_0>0$, $\delta\in(0,1)$, and $q_*>1$ be the constants in~\eqref{stas}, we choose~${R\in(0,\delta)}$}   such that $\bar{\mathbb{B}}_{E_\alpha}(0,R)\subset O_\alpha$ and let 
$$
s(A)<-\varsigma<-\varpi<0\,.
$$ 
Assumption \eqref{Vor3} together with   \cite[II.Lemma~5.1.3]{LQPP}  
ensures that there is a constant~${M\geq 1}$ such that 
\begin{equation}\label{v1as}
\|e^{tA}\|_{\mathcal{L}(E_\theta)}+ t^{\theta-\vartheta}\|e^{tA}\|_{\mathcal{L}(E_\vartheta,E_\theta)}  \le  \frac{M}{4}e^{-\varsigma t}  \,, \quad t>0\,,
\end{equation}
for $\theta,\vartheta\in \{0,\gamma,\alpha,\xi,1\}$ with $\vartheta\le \theta$.
{\tblue We note that for $0\le a<1-b$ and $\omega>0$ we have
\begin{subequations}\label{est4xxx}
\begin{align}
\sup_{t>0}\left(t^{a}\int_0^1(1-s)^{-b}\, e^{-\omega t(1-s)} \, s^{-\mu q_*}\,\rd s\right)
&\le \Big(\sup_{r>0}r^{a}e^{-\omega  r}\Big)\,\mathsf{B}(1-a-b,1-\mu q_*)\,.
\end{align}  
Since $1+\gamma-\alpha-\mu q_*\geq 0$, $q_*>1$, and $\varpi-\varsigma<0$, we may set
\begin{align}
c_1:=\max_{\theta\in\{\alpha,\xi\}}\Big(\sup_{r>0}r^{1+\gamma-\alpha-\mu q_*}e^{(\varpi-\varsigma)  r}\Big)\,\mathsf{B}(\alpha+\mu q_*-\theta,1-\mu q_*)\,.
\end{align} 
\end{subequations}
 We  fix $L\in (0,R)$ such that 
\begin{equation}\label{L}
c_0 c_1M  L^{q_*-1}\le 1\,.
\end{equation} 
Consider} $u_0\in U_\alpha:=\mathbb{B}_{E_\alpha}\big(0,L/M\big)$ and let $u=u(\cdot;u_0)$ be the maximal strong solution to~\eqref{SCP} on~$[0,t^+(u_0))$. Set $\mu:=\xi-\alpha$ and
$$
t_*:=\sup\big\{t\in (0,t^+(u_0))\,;\, \text{$\|u\|_{C_\mu ((0,t],E_\xi)}<L$\, and  \, $\|u\|_{C ([0,t],E_\alpha)}<R $}\big\} \,.
$$
Note that, since $\|u\|_{C_\mu((0,t],E_\xi)}\to 0$ \,  and \, $\|u-u_0\|_{C ([0,t],E_\alpha)}\to0$ as $t\to 0$, it holds that~$t_*>0$. 
Then, for $\theta\in \{\alpha,\xi\}$ and $t\in (0,t_*)$, we obtain from {\tblue    \eqref{stas2},~\eqref{v1as},~and~\eqref{est4xxx}} that
{\tblue \begin{equation}\label{bbb}
\begin{aligned}
t^{\theta-\alpha}\|u(t)\|_\theta&\le \frac{M}{4}\,e^{-\varsigma t}\,\|u_0\|_\alpha +  \frac{c_0M  L^{q_*}  }{4} t^{\theta-\alpha}\int_0^t (t-\tau)^{\gamma-\theta}\,e^{-\varsigma (t-\tau)} \,\tau^{-\mu q_*}\,\rd \tau\\
&\le \frac{M}{4}\,e^{-\varsigma t}\|u_0\|_\alpha +  \frac{c_0M   L^{q_*}}{4} t^{1+\gamma-\alpha-\mu q_*}\int_0^1 (1-s)^{\gamma-\theta}\,e^{-\varsigma t(1-s)} \,s^{-\mu q_*}\,\rd \tau\\
&\le \frac{M}{4} \|u_0\|_\alpha +  \frac{c_0c_1 M   L^{q_*}}{4}\,.
\end{aligned}
\end{equation}}
It now follows from~\eqref{L} and \eqref{bbb} that
$$
\|u\|_{C_\mu((0,t],E_\xi)}\le \frac{L}{2} \qquad    \text{and} \qquad  \|u\|_{C ([0,t],E_\alpha)}\leq\frac{R}{2} 
$$
for each $t\in (0,t_*)$.
 Therefore $t_*=t^+(u_0)$, and we deduce {\tblue  from~\eqref{stas2}} and the definition of~$R$ that 
$$
\underset{t\nearrow t^+(u_0)}{\limsup}\|f(u(t;u_0))\|_{0}<\infty\qquad\text{and} \qquad  \underset{t\nearrow t^+(u_0)}\limsup\,\mathrm{dist}_{E_\alpha}\big( u( t;u_0),\partial  O_\alpha\big)>0\,.
$$
 Theorem~\ref{T:2}~{\bf (iv)(b)} now implies $t^+(u_0)=\infty$. Put
$$
z(t):=\sup_{\tau\in(0,t]} \big(  \|u(\tau)\|_\alpha+\tau^\mu  \|u(\tau)\|_\xi\big)e^{\varpi \tau}\,,\quad t>0\,.
$$
 Given $0<\tau<t$, we then have 
$$
\|u(\tau)\|_\xi^{q_*}\le L^{{q_*}-1} z(t)\, \tau^{-\mu q_*}\,e^{-\varpi \tau}\,,
$$
and together with   {\tblue \eqref{stas2},   \eqref{v1as}, and} the latter estimate   we obtain, analogously to~\eqref{bbb}, that 
\begin{align*}
   z(t)\le \frac{ M}{2} \,\|u_0\|_\alpha + \frac{c_0c_1 M  \,L^{q_*-1}}{2}\, z(t)
\end{align*}
and therefore, by the choice of $L$ from~\eqref{L},
$$
z(t)\le M\|u_0\|_\alpha \,,\quad t>0\,,
$$
that is,
$$
\|u(t)\|_\alpha+t^{\xi-\alpha}\|u(t)\|_\xi\le M\,e^{-\varpi t}\,\|u_0\|_\alpha \,,\quad t>0\,.
$$
This proves Corollary~\ref{C1}.

 \section{Applications of the Abstract Results} \label{Sec:4}

 In this section we first apply the abstract results for the semilinear case -- see Theorem~\ref{T:2} and Corollary~\ref{C1}  -- 
  in the context  of {\tred a coagulation-fragmentation equation with size diffusion (see Example~\ref{Exam4}),}
of an asymptotic model for atmospheric  flows (see Example~\ref{Exam1}),  and of a classical semilinear heat equation (see  Example~\ref{Exam2}).
  In the final part we consider in Example~\ref{Exam3} a quasilinear parabolic evolution problem and establish its local well-posedness in a critical Bessel potential space by virtue of Theorem~\ref{T:0x}.
 
{\tred
\subsection{Example}\label{Exam4} 

The coagulation-fragmentation equation with size diffusion 
\begin{subequations}\label{FD.0}
	\begin{align}
		\partial_t \phi(t,x) & = D \partial_x^2 \phi(t,x) +  \mathcal{F}(\phi(t,\cdot))(x) + \mathcal{K}\big(\phi(t,\cdot),\phi(t,\cdot)\big)(x) \,, 
\label{FD.1} \\
		\phi(t,0) & = 0\,, \qquad 
		\phi(0,x)  = \phi_0(x)\,, 
\label{FD.3} 
	\end{align}
\end{subequations}
for  $t>0$ and $x\in (0,\infty)$, where
\begin{equation*}
	\mathcal{F}(\phi)(x):=- a(x) \phi(x) + \int_x^\infty a(y) b(x,y) \phi(y)\ \mathrm{d}y
\end{equation*}
and 
\begin{equation*}
\mathcal{K}(\phi,\psi)(x):=\frac{1}{2}\int_0^x k(y,x-y) \phi(y)\psi(x-y)\,\mathrm{d}y-\phi(x)\int_0^\infty k(x,y)\psi(y)\,\mathrm{d}y\,,
\end{equation*}
describes the dynamics of the size distribution function $\phi=\phi(t,x)\ge 0$ of particles of size~${x\in (0,\infty)}$ at time $t>0$, where particles modify their sizes according to random fluctuations with rate $D>0$, spontaneous fragmentation with overall fragmentation rate $a\ge 0$ and daughter distribution function $b\ge 0$, and binary coalescence with coagulation kernel~${k\ge 0}$. There is a vast literature
on coagulation-fragmentation equation without size diffusion (i.e. $D=0$), for a summary we refer to \cite{BLL2020a, BLL2020b}. The  inclusion of  size diffusion is done more recently without  \cite{MFJLODS2004,LW_EJAM} and with \cite{OFJM2005,LW_DIE} coagulation (see also the references in the cited papers). Herein we apply  Theorem~\ref{T:2}   and thus extend the well-posedness part of \cite[Theorem~1.1]{LW_DIE} for~\eqref{FD.0}.\\

To this end, we suppose that  $a$ satisfies
\begin{subequations}\label{sda}
\begin{equation}
a\in L_{\infty,loc}([0,\infty))\,, \qquad a\ge 0 \;\text{ a.e. in }\; (0,\infty)\,, \label{A.0}
\end{equation}
while $b$ is a non-negative measurable function on $(0,\infty)^2$ satisfying
\begin{equation}
	\int_0^y x b(x,y)\ \mathrm{d}x = y\,, \qquad y\in (0,\infty)\,, \label{B.0}
\end{equation}
and there is $\delta_2\in (0,1)$ such that
\begin{equation}
	(1-\delta_2) y^2 \ge \int_0^y x^2 b(x,y)\ \mathrm{d}x\,, \qquad y\in (0,\infty)\,. \label{B.10}
\end{equation}
Assumption \eqref{B.0} guarantees that fragmentation conserves the total mass. For the coagulation kernel $k$, we assume that there are $0\le \xi_0 < \xi < 1$, $m>1$, and $k_*>0$ such that
\begin{subequations}\label{K1L}
\begin{equation}\label{K1}
	0\le k(x,y)=k(y,x)\le k_*\frac{\ell(x)\ell(y)}{x+y+(x+y)^m}\,,\qquad (x,y)\in (0,\infty)^2\,,
\end{equation}
where
\begin{equation}
	\ell(x):=\left\{
	\begin{array}{ll} 
		x^{1-2\xi_0}\,, & x\in (0,1)\,,\\
		(1+a(x))^\xi x^m\,, & x>1\,. 
	\end{array}\right. 
\end{equation}
\end{subequations}
\end{subequations}
We refer to~\cite{BLL2020a,LW_DIE} for relevant examples of kernels obeying the assumptions above. We then set
\begin{equation*}
	E_0:=L_1\big( (0,\infty),(x+x^m)\mathrm{d}x \big)
\end{equation*}
and introduce the linear operator
\begin{equation*}
	A \phi:=D\partial_x^2 \phi +\mathcal{F}(\phi)\,,\qquad \phi\in \mathrm{dom}(A)\,,
\end{equation*}
where
\begin{equation*}
	\mathrm{dom}(A) := \{ \phi\in E_0\ :\ \partial_x^2 \phi\in E_0\,, \ a\phi \in E_0\,, \ \phi(0)=0\}\,.
\end{equation*}	
Defining the graph norm of $\phi\in \mathrm{dom}(A)$ by 
\begin{equation*}
	\|\phi\|_{A} := \|\phi\|_{E_0} + \|\partial_x^2 \phi\|_{E_0} + \|a\phi\|_{E_0}\,,
\end{equation*} 
we set
\begin{equation*}
	E_1 := (\mathrm{dom}(A),\|\cdot\|_{A})\,.
\end{equation*} 
It then follows from~\cite[Theorem~1.1]{LW_EJAM} that assumptions~\eqref{sda} imply
\begin{equation}\label{Vor1x}
A\in \mathcal{H}(E_1,E_0)\,,
\end{equation} 
in particular, $E_1$ is a Banach space  that is densely embedded in $E_0$. Finally, we introduce the complex interpolation spaces
\begin{equation*}
	E_\theta := \big[E_0,E_1\big]_{\theta}\,,\qquad \theta\in (0,1)\,, 
\end{equation*}
for which, however, there does not seem to be a precise characterization. 
Nevertheless, it follows under assumptions~\eqref{sda} from~\cite[Lemma~2.2, Lemma~2.4]{LW_DIE} that $\mathcal{K}:E_\xi\times E_\xi\to E_0$ is bilinear. 
  Consequently,  for $f:E_\xi\to E_0$ with $f(\phi):= \mathcal{K}(\phi,\phi)$ there exists a constant $C>0$ such that
\begin{equation}\label{ffff}
\|f(\phi)-f(\psi)\|_{E_0}\leq C(\|\phi\|_{E_\xi}+\|\psi\|_{E_\xi})\|\phi-\psi\|_{E_\xi},\qquad \phi\,,\psi\in E_\xi\,,
\end{equation}
hence $f$ satisfies in particular \eqref{Vor5} with $q=2$ and $\gamma=0$.
Now, we can write \eqref{FD.0} as a Cauchy problem
$$
\phi'=A\phi+f(\phi)\,,\quad t>0\,,\qquad \phi(0)=\phi_0
$$
and the observations above lead us to the following well-posedness result:

\begin{thm}\label{th1}
Suppose \eqref{sda} with $0\le \xi_0 <\xi< 1$ and $m>1$, and consider  $\alpha\in (0,\xi)$ with~${2\xi\le 1+\alpha}$.
Then, given $\phi_0\in E_\alpha$,  the  problem~\eqref{FD.0} has a unique maximal strong solution 
	\begin{equation*}
		\phi = \phi(\cdot;\phi_0)\in C\big([0,t^+(\phi_0)),E_\alpha\big)\cap C\big((0,t^+(\phi_0)),E_1\big)\cap C^1\big((0,t^+(\phi_0)),E_0\big)
	\end{equation*}
	with $t^+(\phi_0)\in (0,\infty]$ such that
	\begin{equation*}
		\lim_{t\to 0} t^{\mu}\|\phi(t)\|_{E_\xi}=0\,,
	\end{equation*} 
where $\mu:=1-\xi$ if $2\xi= 1+\alpha$ and $\mu>\xi-\alpha$ if $2\xi< 1+\alpha$.
The total mass is conserved, that is,
	\begin{equation}
	\int_0^\infty \phi(t,x) x\,\rd x= \int_0^\infty \phi_0(x) x\,\rd x\,, \qquad t\in [0,t^+(\phi_0))\,. \label{M.102}
	\end{equation} 
Moreover, if  $\phi_0\ge 0$, then $\phi(t)\ge 0$ for $t\in [0,t^+(\phi_0))$.  Finally,  the map~${[(t,\phi_0)\mapsto \phi(t;\phi_0)]}$ defines a semiflow on $E_\alpha$. 
\end{thm}

\begin{proof}
The case $2\xi< 1+\alpha$  is treated in \cite[Theorem~1.1]{LW_DIE}. 
The existence and uniqueness result for the case $2\xi= 1+\alpha$ is a consequence of the above observations \eqref{Vor1x}, \eqref{ffff}, and Theorem~\ref{T:2}.
Conservation of mass and positivity of solutions can be shown as in the proof of \cite[Theorem~1.1]{LW_DIE}.
\end{proof}

In fact, additional properties  such as a refined global existence criterion (see \cite[Theorem~1.1]{LW_DIE}) 
and even global existence for positive initial values under more restrictive assumptions on the kernels (see \cite[Theorem~1.3]{LW_DIE}) can be derived. 
Of course, the case~${\alpha> \xi}$ can be included as well as pointed out in the introduction herein.
Moreover, when~${2\xi< 1+\alpha}$, it is actually possible to consider  $\alpha=0$ in Theorem~\ref{th1}, see  \cite[Theorem~1.1]{LW_DIE}.
 Finally, we point out that a maximal regularity approach does not seem to be available for the problem~\eqref{FD.0} in regard of the natural ambient 
 space~${E_0=L_1\big( (0,\infty),(x+x^m)\mathrm{d}x \big)}$.
}

\subsection{Example}\label{Exam1}  We investigate an asymptotic model  derived recently in \cite[Equation (6.18)]{CJ22}  
  describing the propagation of morning glory clouds. More precisely, we consider the system
\begin{subequations}\label{PB}
\begin{equation}\label{EE}
\left.
\arraycolsep=1.4pt
\begin{array}{rcl}
u_t+u u_x+v u_y&=&\nu\Delta u +\eta u+\beta v\\
u_x+v_y&=&0
\end{array}
\right\}\qquad \text{in $\0$\,, \, $t>0$\,,} 
\end{equation}
 in the   horizontally unbounded strip $ \0:=\R\times (0,1)$ or in the periodic strip $\0:=\s\times (0,1)$ (when we consider $2\pi$-periodic solutions with respect to $x$).
The constants in \eqref{EE} satisfy
\[
\nu \in(0,\infty)\,,\qquad  \eta,\,\beta\in\R\,.
\]
The unknown $(u,v)$  is assumed to satisfy  the boundary conditions
\begin{equation}\label{BC}
\left.
\arraycolsep=1.4pt
\begin{array}{rcl}
 u&=&0\qquad \text{on $\p\0$\,, \, $t>0$\,,} \\
 v&=&0\qquad \text{on $\{y=0\}$\,, \, $t>0$}\,,
\end{array}
\right\}
\end{equation}
and $u$ is known  initially
\begin{equation}\label{IC}
u(0)=u_0\,,
\end{equation}
where $u_0:\0\to\R$ is a given function.
\end{subequations}

Similarly as in the context of the viscous primitive equations of large scale ocean and atmosphere dynamics considered in \cite{CaTi07,GuMaRo01}, 
 the  velocity component $v$ can be eliminated from the problem  via the incompressibility condition  in \eqref{EE}  and the boundary condition in~\eqref{BC} 
according to
  \[
  v(z)=\int_0^yv_y(x,s)\,{\rm d}s=-\int_0^y u_x(x,s)\,{\rm d}s = -Tu_x(z)\,,\qquad z=(x,y)\in \Omega\,,
  \]
where, given $w\in L_2(\0)$, we define the mapping $Tw:\0\to\R$ by
  \begin{equation}\label{OpT}
Tw(z):=  \int_0^y w (x,s)\,{\rm d}s\,, \qquad z=(x,y)\in\0\,.
  \end{equation}
 As shown in \cite[Lemma 3]{MR23}, the nonlocal operator $T$ has the following properties:
\begin{lemma}\label{L:1}\phantom{a}
\begin{enumerate}
\item[(i)]   $T\in\kL(L_2(\0))$ satisfies $\|T\|_{\kL(L_2(\0))}\leq 1$;\\[-2ex]
\item[(ii)] $T\in\kL(H^s(\0))$ for all $s\in[0,2]$.
\end{enumerate}
\end{lemma}

Based on this, we may now reformulate \eqref{EE} as a sole evolution equation  for $u$ that takes the form 
 \begin{subequations}\label{PB'}
\begin{equation}\label{EE'}
u_t=\nu\Delta u -u u_x + u_yTu_x+\eta u-\beta Tu_x \qquad \text{in $\0 $\,,\,  $t>0$\,,} 
\end{equation}
subject to the homogeneous Dirichlet boundary condition
\begin{equation}\label{BC'}
 u=0\qquad \text{on $\p\0$\,, \, $t>0$\,,} 
\end{equation}
and the initial condition
\begin{equation}\label{IC'}
u(0)=u_0\,.
\end{equation}
\end{subequations}

This problem has recently been investigated  in  \cite{CJ24}, where it  is shown that, in the presence of an additional source term in \eqref{EE}$_1$, the system possesses traveling-wave solutions.
  In~\cite{MR23}, existence of global weak solution for initial data in $L_2(\0)$ and local-well-posedness in~${H^s_D(\0)}$ with $s\in(1,2)$ has been established in the non-periodic case $\0=\R\times (0,1)$, where
 given~${s\in[0,2]}$, we set 
\begin{equation}\label{Hsd}
H^s_D(\0):=
\left\{
\arraycolsep=1.4pt
\begin{array}{lcl}
H^{s}(\0)\,,&\quad \text{if $s\in[0,1/2)$\,,} \\[1ex]
 \{u \in H^{s}(\0)\,:\, \text{$u=0$ on $\p\0$}\}\,, & \quad\text{if $s\in(1/2,2]$\,.}
\end{array}
\right.
\end{equation}
Subsequently, the existence of global strong solutions in the periodic case $\0=\s\times(0,1)$ for initial data  in $  H^1(\0)\cap L_\infty(\0)$  or in Wiener-like spaces 
 has been investigated in \cite{OB24} under smallness conditions on the initial data and under certain restrictions on the parameters~${\eta,\,\beta,\, \nu}$.
Herein, we consider   both the periodic and the  non-periodic setting simultaneously and use the abstract results in Theorem~\ref{T:2} and Corollary~\ref{C1} to prove that 
problem~\eqref{PB'} generates a semiflow in~$H^1_D(\0)$, which is a critical case that cannot be covered by  the results from \cite{MW_PRSE}. 
Moreover, under suitable assumptions on the parameters $\eta,\,\beta,\, \nu$  it is shown that solutions which are initially small   in~$H^1_D(\0)$ 
 converge at an exponential rate towards the zero solution as time elapses. 

\begin{thm}\label{T:A1}
Let $\nu \in(0,\infty)$ and $\eta,\,\beta\in\R$. 
Then, given  $u_0\in H^1_D(\0)$, there exists a unique maximal strong  solution $u=u(\cdot; u_0)$ to  problem~\eqref{PB'}  with 
\begin{equation*}
u\in C\big((0, t^+), H^2(\0)\big)\cap C^1\big((0,t^+),L_2(\0)\big)\cap C\big([0,t^+),H^1_D(\0)\big)\cap C_{1/4}\big((0,T], H^{3/2}(\0)\big) 
\end{equation*}
for all $0<T < t^+$, where $t^+:=t^+(u_0)\in(0,\infty] $ is the maximal existence time of the solution.
 Moreover, if
 \begin{itemize}
 \item $\0=\R\times(0,1)$ and   $\eta +\cfrac{\vert\beta\vert\pi}{2}+\max\left\{\cfrac{\vert\beta\vert}{2\nu},\pi\right\}\left(\cfrac{\vert\beta\vert}{2}-\pi\nu\right)<0$,
 \end{itemize}
 or
 \begin{itemize} 
\item $\0=\s\times(0,1)$ and $ \eta +\cfrac{\beta^2}{16\nu}<\pi^2\nu$,
 \end{itemize}
 then there exist constants $r,\,\varpi\in (0,1)$ and $M\geq 1$ such that for all $\|u_0\|_{H^1}\leq r$ the solution~${u=u(\cdot; u_0)}$ to~\eqref{PB'}   is globally defined, that is, $t^+(u_0)= \infty$, and
 $$
 \|u(t)\|_{H^1}+t^{1/4}\|u(t)\|_{H^{3/2}}\le M\, e^{-\varpi t}\,\|u_0\|_{H^1} \,,\quad t>0\,.
$$
\end{thm}

The proof of Theorem~\ref{T:A1} is based  upon the auxiliary results of Lemma~\ref{L:2}-Lemma~\ref{L:3} below  and is presented at the end of this section.

 To recast the Cauchy problem~\eqref{PB'}  into an appropriate functional analytic framework, we use the well-known fact 
\[
\Delta_D:=\Delta|_{H^2_D(\0)}\in\mathcal{H}\big( H^2_D(\0), L_2(\0)\big)\,,
\] 
 see e.g.  \cite[Theorem 13.4]{Am84}.
Moreover, letting  $[\cdot,\cdot]_\theta$ be the complex interpolation functor, we have
\begin{equation}\label{f3}
 \big[H^{s_1}_D(\0),H^{s_2}_D(\0)\big]_\theta\doteq H^s_D(\0)\,,\qquad s_1,\, s_2\in[0,2]\,,\quad  s:=(1-\theta )s_1+\theta s_2\neq 1/2\,,
\end{equation}
see  \cite[Theorem 13.3]{Am84}.
 
 We  now set
\[
E_0:=L_2(\0)\,,\qquad E_1:=H^{2}_D(\0)\,,\qquad E_\theta:=[E_0,E_1]_\theta\,,\quad \theta\in[0,1]\,,
\]
and fix
\begin{equation} \label{ddfdf1}
q=2\,,\qquad 0 =:\gamma<\alpha:=\frac{1}{2}<\xi:=\frac34\,.
\end{equation}
We then have
\begin{equation} \label{ddfdf2}
q(\xi-\alpha)=1+\gamma-\alpha\,,
\end{equation}
and, recalling \eqref{f3},
\[
E_\xi=H^{3/2}_D(\0)\hookrightarrow  E_\alpha=H^{1}_D(\0)\,.
\]
 Defining 
\begin{equation*}
Au:=\nu \Delta_D u+\eta u-\beta Tu_x\,, 
\end{equation*}
 the perturbation result \cite[I.Theorem~1.3.1~(ii)]{LQPP} and Lemma~\ref{L:1}  ensure that
 \begin{equation}\label{eq:gen}
 A\in \mathcal{H}( H^2_D(\0), L_2(\0))\,.
 \end{equation}
Introducing the nonlinear function $f: E_\xi\to E_0$  by 
\begin{equation}\label{nonf}
f(u):= u_yTu_x-u u_x\,,
\end{equation}
we may   formulate \eqref{PB'} as the semilinear evolution problem
\begin{equation}\label{SEE}
u'=Au+f(u)\,,\quad t>0\,,\qquad u(0)=u_0\,.
\end{equation}

 We now provide sufficient conditions on the constants $\eta,\, \beta$, and $\nu$    guaranteeing that the semigroup $(e^{tA})_{t\ge 0}$ generated by $A$ on $E_0=L_2(\Omega)$ exhibits exponential decay.

\begin{lemma}[Non-periodic case]\label{L:2}
 If $\0=\R\times(0,1)$, then
\begin{equation}\label{cond:abm}
s(A)\le \eta +\cfrac{\vert\beta\vert\pi}{2}+\max\left\{\cfrac{\vert\beta\vert}{2\nu},\pi\right\}\left(\cfrac{\vert\beta\vert}{2}-\pi\nu\right)\,.
\end{equation}
\end{lemma}
\begin{proof}
Given $ u_0\in E_0$, set $u(t):=e^{tA}u_0$, $t\geq 0$.
Then, since 
\[u\in C^1((0,\infty), E_1)\cap C([0,\infty),E_0)\,,
\] H\"older's inequality, Young's inequality,  Poincar\'e's inequality
 \[\|u\|_2^2 \leq \frac{1}{\pi^2}\|\nabla u\|_2^2,\qquad   u\in H^1_D(\0)\,,\]
for the strip-like domain $\0$ (the optimal Poincar\'e constant of $\0$ being the same as that of the interval $[0,1]$), and Lemma~\ref{L:1}~(i) ensure  for $t>0$  and each $\delta\in (0,1]$ that
\begin{align*}
\frac{\rd}{\rd t}\frac{\|u(t)\|_2^2}{2}&=\int_\0 [u(\nu\Delta u+\eta u-\beta Tu_x)](t)\,\rd z\\[1ex]
&=\int_\0 -\nu|\nabla u|^2(t)+\eta |u|^2(t)-[\beta u(Tu_x)](t)\,\rd z\\[1ex]
&\leq -\nu\|\nabla u(t)\|_2^2+\eta \|u(t)\|_2^2+|\beta|\, \|u(t)\|_2\|Tu_x(t)\|_2\\[1ex]
&\leq  \Big(\eta+\frac{\beta^2}{4 \delta\nu} - (1-\delta)\pi^2\nu\Big)\|u(t)\|_2^2\,. 
\end{align*} 
That is, $\|e^{tA}\|_{\kL(E_0)}\leq e^{\omega(\delta)  t}$ for all $t\geq 0$, where
$$
\omega(\delta):=\eta+\frac{\beta^2}{4 \delta\nu} - (1-\delta)\pi^2\nu\,,\quad \delta\in (0,1]\,.
$$
Consequently, we infer from \cite[Corollary 2.3.2]{L95} that
\begin{align*}
s(A)&\le \min\big\{\omega(\delta)\,;\,\delta\in (0,1]\big\}=\left\{\begin{array}{ll} \omega(1)\,,& d:=\frac{\vert\beta\vert}{2\pi\nu}\ge 1\,,\\[1mm]
\omega(d)\,,& d=\frac{\vert\beta\vert}{2\pi\nu}\le 1\,,\end{array}\right.\\
&= \eta +\cfrac{\vert\beta\vert\pi}{2}+\max\left\{\cfrac{\vert\beta\vert}{2\nu},\pi\right\}\left(\cfrac{\vert\beta\vert}{2}-\pi\nu\right)
\end{align*} 
as is easily checked.
\end{proof}

 In the periodic case  we obtain a slightly weaker constraint on the constants  compared to~\eqref{cond:abm}.
The proof exploits the fact that, since~$A\in \mathcal{H}(E_1,E_0)$ and the embedding~${E_1\hookrightarrow E_0}$ is compact, 
the spectrum $\sigma(A)$ of $A$ consists entirely of isolated  eigenvalues with finite algebraic multiplicities~\cite[Theorem~III.6.29]{Ka95}.
 
\begin{lemma}[Periodic case]\label{L:2'}
If $\0=\s\times(0,1)$ and
\begin{equation}\label{cond:abm'}
 \eta +\frac{\beta^2}{16\nu}< \pi^2\nu\,,
\end{equation}
then the spectral bound $s(A) $ of $A$ is negative.
\end{lemma}
\begin{proof}
 Note that $\lambda\in\mathbb{C}$ is an eigenvalue of $A$ iff there exists a solution  $0\neq u\in H^2_D(\0)$ to
 \begin{equation}\label{BP}
\left.
\arraycolsep=1.4pt
\begin{array}{rcl}
\Delta u+a u+b Tu_x &=&0\quad \text{in $\0$\,,} \\
 u&=&0\quad \text{on $\p\0$\,,}
\end{array}
\right\}
\end{equation}
 where, letting $i$ denote the imaginary unit, we set 
 \[
a=\frac{\eta-\lambda}{\nu}=a_1+ia_2\,,\quad a_1,\, a_2\in\R\,,\qquad\text{and}\qquad b=-\frac{\beta}{\nu}\,. 
 \]
 Then, standard elliptic regularity, e.g. \cite[Theorem 8.13]{GT01}, ensures that  $u\in  C^\infty(\ov\0)$ and therefore, for each $n\in\Z$, the function $u_n:[0,1]\to \mathbb{C}$ with 
 \[
u_n(y):=\int_0^{2\pi} u(x,y)e^{inx}\, \rd x\,,\qquad y\in[0,1]\,,
 \] 
 satisfies $u_n\in C^\infty([0,1])$.
 Moreover, since $\{e^{inx}\,:\,n\in\Z\}$ is  an orthogonal basis of $L_2(\s)$,  problem \eqref{BP} is equivalent to 
  \begin{equation}\label{BPn}
\left.
\arraycolsep=1.4pt
\begin{array}{lcl}
u_n''+(a-n^2) u_n-in b Fu_n = 0\qquad \text{in  $(0,1)$\,,} \\[1ex]
 u_n(0)=u_n(1)=0
\end{array}
\right\} \qquad\text{for all $n\in\Z$\,,}
\end{equation}
where, given $f\in C([0,1])$, we set 
\[
Ff(y):=\int_0^y f(s)\,\rd s, \qquad y\in[0,1]\,.
\]
If $n=0$, the  Sturm-Liouville problem \eqref{BPn} may posses  a nontrivial solution $u_0$ only if~${a\in\R}$.
Moreover, if $a\in\R$, then \eqref{BPn}  has only the trivial solution $u_0=0$ for $a<\pi^2$.

Assume  now $0\neq n\in \Z$ and let $f_n,\, g_n\in C^\infty([0,1],\R)$ denote the real and imaginary parts of $u_n$, that is 
$u_n=f_n+ig_n$. 
Then, $f_n$ and $g_n$ solve the coupled problem
  \begin{equation}\label{BPfg}
\left.
\arraycolsep=1.4pt
\begin{array}{lcl}
f_n''+(a_1-n^2) f_n-a_2g_n+n b Fg_n =0\qquad \text{in  $(0,1)$,} \\
g_n''+(a_1-n^2) g_n+a_2f_n-n b Ff_n =0\qquad \text{in  $(0,1)$,} \\
 f_n(0)=g_n(0)=f_n(1)=g_n(1)=0\,.
\end{array}
\right\} 
\end{equation}
Testing \eqref{BPfg}$_1$  and  \eqref{BPfg}$_2$ with $f_n$ and $g_n$, respectively, and taking the sum of the resulting identities,  we end up with 
\begin{equation}\label{RE}
nb\int_0^1 \Big( f_nFg_n-g_nFf_n\Big)\, \rd y=(n^2-a_1)\int_0^1|u_n|^2\, \rd y+\int_0^1|u_n'|^2\, \rd y\,.
\end{equation}
The boundary conditions \eqref{BPfg}$_3$ enable us to use Wirtinger's inequality for functions, see e.g.~\cite[Section 7.7]{HLP52}, to estimate
\[
\int_0^1|u_n'|^2\, \rd y\geq \pi^2 \int_0^1|u_n|^2\, \rd y\,.
\]
Moreover,  Fubini's theorem and H\"older's inequality   imply
\begin{align*}
nb\int_0^1  \Big(f_nFg_n-g_nFf_n\Big)\, \rd y&\leq |nb|\int_0^1   \Big(|f_n\,|F|g_n|+|g_n|\, F|f_n|\Big)\, \rd y\\
&= |nb|\Big(\int_0^1 |f_n| \, \rd y\Big)\Big(\int_0^1 |g_n| \, \rd y\Big)\leq |nb|\,\|f_n\|_2\,\|g_n\|_2\\
&\leq \frac{ |nb|}{2}\int_0^1|u_n|^2\, \rd y\,.
\end{align*}
These estimates together with \eqref{RE} lead us to
 \begin{equation*}
 \Big(n^2+\pi^2-a_1-\frac{ |nb|}{2}\Big)\int_0^1|u_n|^2\, \rd y\leq 0\,.
\end{equation*}
It is however not difficult to verify that the condition \eqref{cond:abm'} ensures that for all $\lambda\in \mathbb{C}$ with~$\re\lambda\geq 0$ and~$n\in\Z$ it holds that    $n^2+\pi^2>a_1+ |nb|/2$, hence 
$u_n=0$ for all $n\in\Z$. Consequently,  all eigenvalues of $A$ have a negative real part and the claim follows.
\end{proof}

 Regarding the nonlinearity  $  f(u)= u_yTu_x-u u_x$  defined in \eqref{nonf} we establish the following result.  Recall that $E_\xi=H_D^{3/2}(\Omega)$.

\begin{lemma}\label{L:3}
The function  $f:E_\xi\to E_0$  satisfies $f(0)=0$ and is locally Lipschitz continuous in the sense there exists a constant $C>0$ such that for all $u_1,\, u_2\in E_\xi$ we have
\begin{equation}\label{Eq:LE}
\|f(u_1)-f(u_2)\|_{E_0}\leq C(\|u_1\|_{E_\xi}+\|u_2\|_{E_\xi})\|u_1-u_2\|_{E_\xi}\,.
\end{equation}
\end{lemma}
\begin{proof} Lemma~\ref{L:1}~(ii) together with  the embedding $H^{1/2}(\0) \hookrightarrow L_4(\0) $  yields
\begin{align*}
\|f(u)\|_2& \leq  \|u_y\|_4\|Tu_x\|_4+\|u\|_4 \|u_x\|_4\\
&\leq C(\|u_y\|_{H^{1/2}}\|Tu_x\|_{H^{1/2}}+\|u\|_{H^{1/2}} \|u_x\|_{H^{1/2}})\\
&\leq C\|u\|_{E_\xi}^2\,,
\end{align*}
and the claim  is an immediate consequence of this estimate.
\end{proof}

We are now in a position to establish Theorem~\ref{T:A1}.

\begin{proof}[Proof of Theorem~\ref{T:A1}]
Recalling~\eqref{Hsd}-\eqref{ddfdf2}, \eqref{eq:gen},  Lemma~\ref{L:2}, Lemma~\ref{L:2'}, and Lemma~\ref{L:3}, we may apply  Theorem~\ref{T:2} and Corollary~\ref{C1}
  to the semilinear evolution problem~\eqref{SEE}  and obtain the desired claim.
\end{proof} 

\subsection{Example }\label{Exam2} We revisit a classical semilinear heat equation investigated previously, e.g., in~\cite{We81, Fu66}  of the form
\begin{subequations}\label{FW}
\begin{equation}\label{FW1}
 \partial_t u=\Delta u +|u|^{\kappa-1}u\qquad \text{in $\0  $\,,\, $t>0 $\,,} 
\end{equation}
 with homogeneous Dirichlet boundary conditions
\begin{equation}\label{FW2}
 u=0\qquad \text{on $\p\0$\,, \, $t>0$\,,} \\
\end{equation}
and  a prescribed initial datum
\begin{equation}\label{FW3}
u(0)=u_0\,,
\end{equation}
where $u_0:\0\to\R$ is a given function, $\0\subset\R^n$ with $n\geq 1$ is a smooth bounded domain, and~${\kappa>1+2/n}$.
\end{subequations}

 As pointed out in \cite{PSW18}, problem~\eqref{FW} features a scaling invariance. 
 Indeed, if $u$ is a solution to \eqref{FW1} on $\R^n$, then for each $\lambda>0$ the function
\[
u_\lambda(t,x):=\lambda^{\tfrac{1}{\kappa-1}}u(\lambda t, \sqrt{\lambda}x)
\]
is again a solution to \eqref{FW1} and its seminorm
$$
\|\mathcal{F}^{-1}(\vert\cdot\vert^\mathsf{s}\mathcal{F}u_\lambda)\|_{L_p}
$$
is independent of $\lambda$ provided that $\mathsf{s}=s_c:=\frac{n}{p}-\frac{2}{\kappa-1}$.
This scaling invariance identifies~$H^{s_c}_{p}$ as a critical scaling invariant space for the evolution equation~\eqref{FW1}. We shall prove that problem~\eqref{FW} defines in fact a semiflow on the phase space $H^{s_c}_{p,D}(\0):=cl_{H_p^{s_c}} C_c^\infty(\0)$. More precisely:

\begin{thm}\label{T:A2}
Let $n\geq 1$, $\kappa>1+2/n$,  choose
\begin{equation*}
\max\Big\{1,\,\frac{n(\kappa-1)}{2\kappa}\Big\}<p<\frac{n(\kappa-1)}{2} \,,\qquad p\neq \frac{(n-1)(\kappa-1)}{2}\,,
\end{equation*}
and set
\begin{equation*}\label{bbp}
 0<s_c:=\frac{n}{p}-\frac{2}{\kappa-1}<s:=\frac{n(\kappa-1)}{p\kappa}<2\,,\qquad  \mu:=\frac{1}{\kappa-1}-\frac{n}{2p\kappa}\in(0,1)\,.
\end{equation*}
Then, given  $u_0\in H^{s_c}_{p,D}(\0)$, there exists a unique maximal strong  solution $u=u(\cdot; u_0)$ to  problem~\eqref{FW}  such that
\begin{equation*}
u\in C\big((0,t^+),H^2_{p}(\0)\big)\cap C^1\big((0,t^+),L_p(\0)\big)\cap C\big([0,t^+),H^{s_c}_{p,D}(\0)\big)\cap C_{\mu}\big((0,T], H^{s}_{p}(\0)\big) 
\end{equation*}
for all $0<T < t^+$, where $t^+:=t^+(u_0)\in(0,\infty] $ is the maximal existence time of the solution.
 Moreover, there exist constants $r,\,\varpi\in (0,1)$ and $M\geq 1$ such that for all $\|u_0\|_{H^{s_c}_{p}}\leq r$ the solution~${u=u(\cdot; u_0)}$ to~\eqref{FW}
    is globally defined, that is, $t^+(u_0)=\infty$, and
 $$
 \|u(t)\|_{H^{s_c}_{p}}+t^{\mu}\|u(t)\|_{H^{s}_p}\le M\, e^{-\varpi t}\,\|u_0\|_{H^{s_c}_p} \,,\quad t>0\,.
$$
\end{thm}

\begin{proof}
 In order to recast \eqref{FW} into a suitable functional analytic  framework we define
$$
E_0:=L_p(\Omega)\,,\qquad E_1:=H_{p,D}^{2}(\Omega):=\{v\in H_{p}^{2}(\Omega)\,:\,   v=0 \text{ on } \partial\Omega\}\,,
$$
and set $E_\theta:=[E_0,E_1]_\theta$ for  $\theta\in[0,1]$, where  $[\cdot,\cdot]_\theta$ is  once more the  complex interpolation functor.
We then infer from \cite[\S 4]{Amann_Teubner}   that 
\begin{equation}\label{Lapla}
\Delta_D\in \mathcal{H}(E_1,E_0)
\end{equation}
and recall that
\begin{equation}\label{f2}
 E_\theta\doteq H_{p,D}^{2\theta}(\Omega)=\left\{\begin{array}{ll} \{v\in H_{p}^{2\theta}(\Omega) \,:\,    v=0 \text{ on } 
 \partial\Omega\}\,, &\frac{1}{p}<2\theta<2 \,,\\[3pt]
	 H_{p}^{2\theta}(\Omega)\,, & 0< 2\theta<\frac{1}{p}\,.\end{array} \right.
\end{equation}
We note that the constants $s$ and $s_c$   defined in \eqref{bbp} are both different from $1/p$. 
We set
\begin{equation} \label{fffv0} 
q:=\kappa>1\qquad\text{and}\qquad 0 =:\gamma<\alpha:=\frac{s_c}{2}<\xi:=\frac{s}{2}<1\,,
\end{equation}
with
\begin{equation}\label{fffv} 
q(\xi-\alpha)=1+\gamma-\alpha\,,
\end{equation}
and infer from \eqref{f2} that $E_\alpha=H^{s_c}_{p,D}(\0)$ and $E_\xi=H^{s}_{p,D}(\0)$.
With this notation we may formulate \eqref{FW} as the semilinear parabolic problem
\begin{equation}\label{SEE2}
u'=\Delta_D u+f(u)\,,\quad t>0\,,\qquad u(0)=u_0\,,
\end{equation} 
where $f:E_\xi\to E_0$ is defined by 
\[
f(u):=|u|^{\kappa-1}u\,.
\] 
In view of the Sobolev embedding $H^{s}_{p}(\0)\hookrightarrow L_{p\kappa}(\0)$ and making use of the inequality
\begin{equation}\label{eq:aux}
\big|\,|x|^{\kappa-1}x-|y|^{\kappa-1}y\,\big|\leq \kappa\big(|x|^{\kappa-1}+|y|^{\kappa-1})|x-y|\,,\qquad x,\, y\in\R\,,
\end{equation}
  we obtain via H\"older's inequality
\begin{align*}
\|f(u)-f(v)\|_{E_0}\leq \kappa (\|u\|_{E_\xi}^{\kappa-1}+\|v\|_{E_\xi}^{\kappa-1})\|u-v\|_{E_\xi},\qquad u,\, v\in E_\xi\,.
\end{align*}
This estimate, combined with \eqref{Lapla}-\eqref{fffv}  and Remark~\ref{R00}, allows us to apply Theorem~\ref{T:2}  to~\eqref{SEE2} and
 deduce the local well-posedness of this evolution problem in~$ H^{s_c}_{p,D}(\0)$.
Moreover,  since $s(\Delta_D)<0$, Corollary~\ref{C1} finally provides 
the exponential stability of the zero solution in the   critical space~$ H^{s_c}_{p,D}(\0)$.
\end{proof}

\subsection{Example }\label{Exam3} To exemplify also the strength of Theorem~\ref{T:0x}, we consider  the following quasilinear evolution equation from  \cite{PSW18, QS19}:
\begin{subequations}\label{Ex3}
\begin{equation}\label{Ex3a}
 \partial_tu={\rm div}(a(u)\nabla u) +|\nabla u|^{\kappa}\qquad \text{in $\0  $\,,\, $t>0 $\,,} 
\end{equation}
subject to homogeneous Neumann boundary conditions
\begin{equation}\label{Ex3b}
 \p_\nu u=0\qquad \text{on $\p\0$\,, \, $t>0$\,,} \\
\end{equation}
and initial condition
\begin{equation}\label{Ex3c}
u(0)=u_0\,,
\end{equation}
\end{subequations}
where ${\kappa>3}$, $u_0:\0\to\R$ is a given function,  and $\0\subset\R^n$  with $n\geq 1$ is a smooth bounded domain with outward unit normal $\nu$.

As pointed in \cite{PSW18}, problem~\eqref{Ex3} is scaling invariant if $a$ is a constant function. Indeed, in this case, if $u$ is a solution to \eqref{Ex3a} on $\R^n$, then for each $\lambda>0$ the function
\[
u_\lambda(t,x):=\lambda^{-\tfrac{\kappa-2}{2(\kappa-1)}}u(\lambda t, \sqrt{\lambda}x)
\]
is again a solution to \eqref{Ex3a} which identifies the phase space~$H^{s_c}_{p}$ with $s_c:=\frac{n}{p}+\frac{\kappa-2}{\kappa-1}$ as a critical scaling invariant space for the evolution equation~\eqref{Ex3a}. 
In fact, setting 
\begin{equation*}
 H_{p,N}^{s}(\Omega):=\left\{\begin{array}{ll} \{v\in H_{p}^{s}(\Omega) \,:\, \p_\nu v=0 \text{ on } 
 \partial\Omega\}\,, &1+\frac{1}{p}<s\leq 2 \,,\\[3pt]
	 H_{p}^{s}(\Omega)\,, & -1+\frac{1}{p}< s<1 +\frac{1}{p}\,,\end{array} \right.
\end{equation*}
we shall establish the following local well-posedness result in $H^{s_c}_{p,N}(\0)$ for the evolution problem~\eqref{Ex3}.

\begin{thm}\label{T:A3}
Let $\kappa>3$ and let  $a\in C^{1}(\R)$ be a  strictly positive function
with uniformly Lipschitz continuous derivative.
We choose $p\in (2n,(\kappa-1)n) $  with $p \neq (n-1)(\kappa-1)$ and~$\tau\in(0,1)$ such that
\begin{equation*}
\frac{1}{2}<2\tau<1-\frac{n}{p} 
\end{equation*}
and set
\begin{equation*}
 0<\bar s:=2\tau+ \frac{n}{p}<s_c:=\frac{n}{p}+\frac{\kappa-2}{\kappa-1}<s:=1+\frac{n(\kappa-1)}{p\kappa}<  2-2\tau\,,
\end{equation*}
as well as
\begin{equation*}
 \mu:=\frac{1}{2(\kappa-1)}-\frac{n}{2p\kappa}\in(0,1)\,,\qquad  \vartheta:=\frac{\kappa-2}{2(\kappa-1)}-\tau\in(0,1)\,.
\end{equation*}
Then, given  $u_0\in H^{s_c}_{p,N}(\0)$, there exists a unique maximal strong  solution $u=u(\cdot; u_0)$ to  problem~\eqref{Ex3}  such that
\begin{equation*}
\begin{aligned}
u&\in C\big((0,t^+), H^{2-2\tau}_{p}(\0)\big)\cap C^1\big((0,t^+), H^{-2\tau}_{p}(\0)\big)\cap C\big([0,t^+),H^{s_c}_{p,N}(\0)\big)\\[1ex]
 &\quad\cap  C^{\vartheta}\big([0,t^+(u_0)), H^{\bar s}_{p}(\0)\big)\cap C_{\mu}\big((0,T], H^{s}_{p}(\0)\big) 
\end{aligned}
\end{equation*}
for all $0<T < t^+$, where $t^+:=t^+(u_0)\in(0,\infty] $ is the maximal existence time of the solution.
\end{thm}

\begin{proof}
We introduce
$$
F_0:=L_p(\Omega)\,,\qquad F_1:= W_{p,N}^{2}(\Omega)=H_{p,N}^{2}(\Omega)=\{v\in H_{p}^{2}(\Omega)\,:\,  \p_\nu u=0 \text{ on } \partial\Omega\}\,, 
$$
and note  from \cite[\S 4]{Amann_Teubner} that 
$$
B_0:=\Delta_N:=\Delta|_{W_{p,N}^{2}(\Omega)}\in \mathcal{H}\big(W_{p,N}^{2}(\Omega),L_p(\Omega)\big)\,.
$$
Let
$$
\big\{(F_\theta,B_\theta)\,:\, -1\le \theta<\infty\big\}\ 
$$
be the interpolation-extrapolation scale generated by $(F_0,B_0)$ and the 
complex interpolation functor~$[\cdot,\cdot]_\theta$ (see \cite[\S 6]{Amann_Teubner} and \cite[\S V.1]{LQPP}), that is,
\begin{equation}\label{f2x}
B_\theta\in \mathcal{H}(F_{1+\theta},F_\theta)\,,\quad -1\le \theta<\infty\,,
\end{equation}
with (see \cite[Theorem~7.1; Equation (7.5)]{Amann_Teubner})
\begin{equation}\label{f2N}
 F_\theta\doteq H_{p,N}^{2\theta}(\Omega)\,,\qquad   2\theta\in\Big(-1+\frac{1}{p},2\Big]\setminus\Big\{1+\frac{1}{p}\Big\}\,.
\end{equation}
Moreover, since
$\Delta_N-1$ has bounded imaginary powers (see~\cite[III.~Examples 4.7.3~(d)]{LQPP}), we infer from \cite[Remarks~6.1~(d)]{Amann_Teubner} that
\begin{equation}\label{f3N}
 [F_\beta,F_\alpha]_\theta\doteq F_{(1-\theta)\beta+\theta\alpha}\,,\qquad -1\leq \beta<\alpha\,,\quad \theta\in(0,1)\,.
\end{equation}
With  $p$ and $\tau$ fixed  as in the statement, we set
$$
E_\theta:= H^{2\theta-2\tau}_{p,N}(\Omega)\,,\qquad 2\tau+1+\frac{1}{p}\neq 2\theta\in[0,2]\,,
$$
 and point out that  $E_\theta=[E_0,E_1]_\theta$.
Note that none of the constants $\bar s$, $s$, and $s_c$  fixed in  the statement is equal to $1+1/p$. 
We set $q:=\kappa>3$,
\begin{equation} \label{bnnn1} 
 0<\gamma:=\tau<\beta:= \tau+ \frac{\bar s }{2}<\alpha:=\tau+ \frac{s_c}{2}<\xi:=\tau +\frac{s}{2}<1\,,
\end{equation}
and  observe that
\begin{equation}\label{bnnn2} 
q(\xi-\alpha)=1+\gamma-\alpha\,,
\end{equation}
and,  in view of \eqref{f2N},
\[
E_\xi=H^{s}_{p,N}(\0)\hookrightarrow E_\alpha=H^{s_c}_{p,N}(\0)\hookrightarrow E_\beta=H^{\bar s}_{p,N}(\0)\hookrightarrow E_\gamma=L_{p}(\0)\,.
\]
With this notation we may formulate \eqref{Ex3} as the quasilinear  parabolic problem
\begin{equation}\label{QEE2}
u'=A(u)u+f(u)\,,\quad t>0\,,\qquad u(0)=u_0\,,
\end{equation} 
where $A:E_\beta\to\kL(E_1,E_0)$ and $f:E_\xi\to E_\gamma$ are  defined by 
\[
A(u)v:={\rm div} (a(u)\nabla v)\,,\qquad  v\in E_1\,,\quad u\in E_\beta\,,
\]
and
\[
f(u):=|\nabla u|^{\kappa}\,,\qquad  u\in E_\xi\,.
\] 
Noticing that $ E_\xi= H^{s}_{p,N}(\0)\hookrightarrow H^1_{p\kappa}(\0)$, H\"older's inequality together with  inequality \eqref{eq:aux} yields
\begin{align}\label{Lips}
\|f(u)-f(v)\|_{E_\gamma}\leq \kappa (\|u\|_{E_\xi}^{\kappa-1}+\|v\|_{E_\xi}^{\kappa-1})\|u-v\|_{E_\xi}\,,\qquad u,\, v\in E_\xi\,.
\end{align}
In order to deal with the operator $A$, we choose 
$$2\e\in(0\,, 4\tau-1)\,$$
and recall from \cite[Equations~(5.2)-(5.6)]{Amann_Teubner} the embeddings
 \begin{equation}\label{f4}
 H_{p}^{\bar s}(\Omega)\hookrightarrow  W_{p}^{\bar s-\e}(\Omega) \hookrightarrow   H_{p}^{\bar s-2\e}(\Omega) \hookrightarrow   H_{p}^{1-2\tau }(\Omega)\,.
\end{equation}
Since~${1-2\tau>n/p}$, all these  spaces are algebras with respect to pointwise multiplication.
We then deduce from  \cite[Lemma 4.1]{MW_PRSE} that 
\[
[u\mapsto a(u)]\in C^{1-}(H^{\bar s}_p(\0),H^{\bar s-2\e}_p(\0))\,,
\]
 from which it readily follows that $A\in C^{1-}(E_\beta,\kL(E_1,E_0))$. 
 Moreover, given $u\in E_\beta$, note that \mbox{$a(u)\in H^{\bar s-2\e}_p(\0)\hookrightarrow C^\rho(\ov\Omega)$} with  $\rho:=2(\tau-\e)>1-2\tau>0$  and thus $A(u)\in \kH(E_1,E_0)$
due to  \cite[Theorem~8.5]{Amann_Teubner} and therefore $A\in C^{1-}(E_\beta,\kH(E_1,E_0))$.

The latter property together with \eqref{f3N}, \eqref{bnnn1}, \eqref{bnnn2}, and  \eqref{Lips} enables us to apply Theorem~\ref{T:0x} 
in the context of~\eqref{QEE2} to deduce the local well-posedness of this evolution problem in $ H^{s_c}_{p,N}(\0)$.
\end{proof}

\bibliographystyle{siam}
\bibliography{Literature}

\begin{thebibliography}{10}

\bibitem{AT88}
{\sc P.~Acquistapace and B.~Terreni}, {\em On quasilinear parabolic systems},
  Math. Ann., 282 (1988), p.~315–335.

\bibitem{OB24}
{\sc D.~Alonso-Or\'{a}n and R.~Granero-Belinchón}, {\em Well-posedness and
  decay for a nonlinear propagation wave model in atmospheric flows}.
\newblock arXiv:2401.09564.

\bibitem{Am84}
{\sc H.~Amann}, {\em Existence and regularity for semilinear parabolic
  evolution equations}, Ann. Scuola Norm. Sup. Pisa Cl. Sci. (4), 11 (1984),
  pp.~593--676.

\bibitem{Am88}
\leavevmode\vrule height 2pt depth -1.6pt width 23pt, {\em {Dynamic theory of
  quasilinear parabolic equations. {I}. {A}bstract evolution equations}},
  Nonlinear Anal., 12 (1988), p.~895–919.

\bibitem{Amann_Teubner}
\leavevmode\vrule height 2pt depth -1.6pt width 23pt, {\em {Nonhomogeneous
  linear and quasilinear elliptic and parabolic boundary value problems}}, in
  {Function spaces, differential operators and nonlinear analysis
  ({F}riedrichroda, 1992)}, vol.~133 of {Teubner-Texte Math.}, Teubner,
  Stuttgart, 1993, p.~9–126.

\bibitem{LQPP}
\leavevmode\vrule height 2pt depth -1.6pt width 23pt, {\em Linear and
  Quasilinear Parabolic Problems. {V}ol. {I}}, vol.~89 of Monographs in
  Mathematics, Birkh\"{a}user Boston, Inc., Boston, MA, 1995.
\newblock Abstract linear theory.

\bibitem{An90}
{\sc S.~B. Angenent}, {\em {Nonlinear analytic semiflows}}, Proc. Roy. Soc.
  Edinburgh Sect. A, 115 (1990), pp.~91--107.

\bibitem{BLL2020a}
{\sc J.~Banasiak, W.~Lamb, and {\relax Ph}.~Lauren\c{c}ot}, {\em Analytic
  Methods for Coagulation-Fragmentation Models. {V}ol. {I}}, Monographs and
  Research Notes in Mathematics, CRC Press, Boca Raton, FL, 2020.

\bibitem{BLL2020b}
\leavevmode\vrule height 2pt depth -1.6pt width 23pt, {\em Analytic Methods for
  Coagulation-Fragmentation Models, {V}ol. {II}}, Monographs and Research Notes
  in Mathematics, CRC Press, Boca Raton, FL, 2020.

\bibitem{CaTi07}
{\sc C.~Cao and E.~S. Titi}, {\em Global well-posedness of the
  three-dimensional viscous primitive equations of large scale ocean and
  atmosphere dynamics}, Ann. Math, 166 (2007), pp.~245--267.

\bibitem{ClementLi}
{\sc P.~Cl\'{e}ment and S.~Li}, {\em Abstract parabolic quasilinear equations
  and application to a groundwater flow problem}, Adv. Math. Sci. Appl., 3
  (1993/94), pp.~17--32.

\bibitem{CS01}
{\sc {\relax Ph}.~Cl\'{e}ment and G.~Simonett}, {\em Maximal regularity in
  continuous interpolation spaces and quasilinear parabolic equations}, J.
  Evol. Equ., 1 (2001), pp.~39--67.

\bibitem{CJ22}
{\sc A.~Constantin and R.~S. Johnson}, {\em On the propagation of nonlinear
  waves in the atmosphere}, Proc. A, 478 (2022), p.~20210895.

\bibitem{CJ24}
\leavevmode\vrule height 2pt depth -1.6pt width 23pt, {\em Atmospheric undular
  bores}, Math. Ann., 388 (2024), pp.~4011--4036.

\bibitem{DaP96}
{\sc G.~{Da Prato}}, {\em Fully nonlinear equations by linearization and
  maximal regularity, and applications}, in Partial differential equations and
  functional analysis, vol.~22 of Progr. Nonlinear Differential Equations
  Appl., Birkhäuser Boston, Boston, MA, 1996, p.~80–92.

\bibitem{DaPG79}
{\sc G.~{Da Prato} and P.~Grisvard}, {\em Equations d'évolution abstraites non
  linéaires de type parabolique}, Ann. Mat. Pura Appl. (4), 120 (1979),
  p.~329–396.

\bibitem{DaPL88}
{\sc G.~{Da Prato} and A.~Lunardi}, {\em Stability, instability and center
  manifold theorem for fully nonlinear autonomous parabolic equations in
  {B}anach space}, Arch. Rational Mech. Anal., 101 (1988), p.~115–141.

\bibitem{Fu66}
{\sc H.~Fujita}, {\em On the blowing up of solutions of the {C}auchy problem
  for {$u\sb{t}=\Delta u+u\sp{1+\alpha }$}}, J. Fac. Sci. Univ. Tokyo Sect. I,
  13 (1966), pp.~109--124.

\bibitem{GT01}
{\sc D.~Gilbarg and N.~S. Trudinger}, {\em {Elliptic Partial Differential
  Equations of Second Order}}, Springer Verlag, 2001.

\bibitem{G88}
{\sc D.~Guidetti}, {\em Convergence to a stationary state and stability for
  solutions of quasilinear parabolic equations}, Ann. Mat. Pura Appl. (4), 151
  (1988), p.~331–358.

\bibitem{GuMaRo01}
{\sc F.~Guill{\'e}n-Gonz{\'a}lez, N.~Masmoudi, and M.~A.
  Rodr{\'i}guez-Bellido}, {\em {Anisotropic estimates and strong solutions of
  the primitive equations}}, Differential and Integral Equations, 14 (2001),
  pp.~1381--1408.

\bibitem{HLP52}
{\sc G.~H. Hardy, J.~E. Littlewood, and G.~P\'{o}lya}, {\em Inequalities},
  Cambridge, at the University Press,, 1952.
\newblock 2nd ed.

\bibitem{HvNWV_III}
{\sc T.~Hyt\"onen, J.~van Neerven, M.~Veraar, and L.~Weis}, {\em Analysis in
  {B}anach Spaces. {V}ol. {III}. {H}armonic Analysis and Spectral Theory},
  vol.~76 of Ergebnisse der Mathematik und ihrer Grenzgebiete. 3. Folge. A
  Series of Modern Surveys in Mathematics [Results in Mathematics and Related
  Areas. 3rd Series. A Series of Modern Surveys in Mathematics], Springer,
  Cham, [2023] \copyright2023.

\bibitem{Ka95}
{\sc T.~Kato}, {\em {Perturbation Theory for Linear Operators}}, Springer
  Verlag, Berlin, 1995.

\bibitem{LW_EJAM}
{\sc {\relax Ph}.~Lauren\c{c}ot and {\relax Ch}.~Walker}, {\em The
  fragmentation equation with size diffusion: Well-posedness and long-term
  behavior}, European J. Appl. Math.,  (2022), pp.~1--34.
\newblock arXiv: 2104.14798.

\bibitem{LW_DIE}
\leavevmode\vrule height 2pt depth -1.6pt width 23pt, {\em Well-posedness of
  the coagulation-fragmentation equation with size diffusion}, Differential
  Integral Equations, 35 (2022), pp.~211--240.

\bibitem{CPW10}
{\sc J.~LeCrone, J.~Pr\"{u}ss, and M.~Wilke}, {\em On quasilinear parabolic
  evolution equations in weighted {$L_p$}-spaces {II}}, J. Evol. Equ., 14
  (2014), pp.~509--533.

\bibitem{LeCroneSimonett20}
{\sc J.~LeCrone and G.~Simonett}, {\em On quasilinear parabolic equations and
  continuous maximal regularity}, Evol. Equ. Control Theory, 9 (2020),
  pp.~61--86.

\bibitem{Lu84}
{\sc A.~Lunardi}, {\em Abstract quasilinear parabolic equations}, Math. Ann.,
  267 (1984), p.~395–415.

\bibitem{Lu85}
\leavevmode\vrule height 2pt depth -1.6pt width 23pt, {\em Asymptotic
  exponential stability in quasilinear parabolic equations}, Nonlinear Anal., 9
  (1985), p.~563–586.

\bibitem{Lu85b}
\leavevmode\vrule height 2pt depth -1.6pt width 23pt, {\em Global solutions of
  abstract quasilinear parabolic equations}, J. Differential Equations, 58
  (1985), p.~228–242.

\bibitem{Lu87}
\leavevmode\vrule height 2pt depth -1.6pt width 23pt, {\em On the local
  dynamical system associated to a fully nonlinear abstract parabolic
  equation}, in Nonlinear analysis and applications ({A}rlington, {T}ex.,
  1986), vol.~109 of Lecture Notes in Pure and Appl. Math., Dekker, New York,
  1987, p.~319–326.

\bibitem{L95}
\leavevmode\vrule height 2pt depth -1.6pt width 23pt, {\em {Analytic Semigroups
  and Optimal Regularity in Parabolic Problems}}, {Progress in Nonlinear
  Differential Equations and their Applications, 16}, Birkhäuser Verlag,
  Basel, 1995.

\bibitem{MFJLODS2004}
{\sc J.~Mathiesen, J.~Ferkinghoff-Borg, M.~H. Jensen, M.~Levinsen, P.~Olesen,
  D.~Dahl-Jensen, and A.~Svenson}, {\em Dynamics of crystal formation in the
  greenland {NorthGRIP} ice core}, J. Glaciol., 50 (2004), pp.~325--328.

\bibitem{MR23}
{\sc B.-V. Matioc and L.~Roberti}, {\em Weak and classical solutions to an
  asymptotic model for atmospheric flows}, J. Differential Equations, 367
  (2023), pp.~603--624.

\bibitem{MW_MOFM20}
{\sc B.-V. Matioc and {\relax Ch}.~Walker}, {\em On the principle of linearized
  stability in interpolation spaces for quasilinear evolution equations},
  Monatsh. Math., 191 (2020), pp.~615--634.

\bibitem{MW_PRSE}
\leavevmode\vrule height 2pt depth -1.6pt width 23pt, {\em Well-posedness of
  quasilinear parabolic equations in time-weighted spaces}, Proc. Roy. Soc.
  Edinburgh Sect. A, to appear (preprint (2023)).

\bibitem{OFJM2005}
{\sc P.~Olesen, J.~Ferkinghoff-Borg, M.~H. Jensen, and J.~Mathiesen}, {\em
  Diffusion, fragmentation, and coagulation processes: Analytical and numerical
  results}, Phys. Rev. E, 72 (2005), p.~031103.

\bibitem{PW17}
{\sc J.~Pr\"{u}ss and M.~Wilke}, {\em Addendum to the paper ``{O}n quasilinear
  parabolic evolution equations in weighted {$L_p$}-spaces {II}''
  [{MR}3250797]}, J. Evol. Equ., 17 (2017), pp.~1381--1388.

\bibitem{P02}
{\sc J.~Prüss}, {\em Maximal regularity for evolution equations in
  {$L_p$}-spaces}, Conf. Semin. Mat. Univ. Bari,  (2002), p.~1–39 (2003).

\bibitem{PS16}
{\sc J.~Prüss and G.~Simonett}, {\em Moving Interfaces and Quasilinear
  Parabolic Evolution Equations}, vol.~105 of Monographs in Mathematics,
  Birkhäuser/Springer, [Cham], 2016.

\bibitem{PSW18}
{\sc J.~Prüss, G.~Simonett, and M.~Wilke}, {\em Critical spaces for
  quasilinear parabolic evolution equations and applications}, J. Differential
  Equations, 264 (2018), p.~2028–2074.

\bibitem{QS19}
{\sc P.~Quittner and P.~Souplet}, {\em Superlinear Parabolic Problems},
  Birkh\"auser Advanced Texts: Basler Lehrb\"ucher. [Birkh\"auser Advanced
  Texts: Basel Textbooks], Birkh\"auser/Springer, Cham, second~ed., 2019.
\newblock Blow-up, global existence and steady states.

\bibitem{We81}
{\sc F.~B. Weissler}, {\em Existence and nonexistence of global solutions for a
  semilinear heat equation}, Israel J. Math., 38 (1981), pp.~29--40.

\end{thebibliography}

\end{document}